\documentclass[11pt,reqno]{amsart}
\usepackage{amsmath,amssymb,amsfonts,amsthm}
\usepackage{fullpage}
\usepackage{graphicx}
\usepackage{fixmath}
\usepackage{comment}
\usepackage{hhline}
\usepackage[latin1]{inputenc}
\usepackage[margin=2.5cm]{geometry}

\usepackage{xcolor}
\usepackage{caption}

\usepackage{url}
\usepackage{enumitem}

\usepackage{setspace} 

\numberwithin{equation}{section}

\newcommand\andd{\text{ and }}

\newcommand\cA{{\mathcal A}}

\newcommand\cC{{\mathcal C}}

\newcommand\cF{{\mathcal F}}

\newcommand\cO{{\mathcal O}}
\newcommand\cP{{\mathcal P}}

\newcommand\cV{{\mathcal V}}

\newcommand\bb{{\boldsymbol{b}}}

\newcommand\bd{{\boldsymbol{d}}}
\newcommand\be{{\boldsymbol{e}}}
\newcommand\bbf{{\boldsymbol{f}}}
\newcommand\bg{{\boldsymbol{g}}}
\newcommand\bh{{\boldsymbol{h}}}

\newcommand\bk{{\boldsymbol{k}}}

\newcommand\bm{{\boldsymbol{m}}}
\newcommand\bn{{\boldsymbol{n}}}

\newcommand\bp{{\boldsymbol{p}}}

\newcommand\br{{\boldsymbol{r}}}
\newcommand\bs{{\boldsymbol{s}}}
\newcommand\bt{{\boldsymbol{t}}}
\newcommand\bu{{\boldsymbol{u}}}
\newcommand\bv{{\boldsymbol{v}}}
\newcommand\bw{{\boldsymbol{w}}}

\newcommand\bZr{{\boldsymbol{0}}}
\newcommand\C{\mathbb{C}}

\newcommand\R{\mathbb{R}}
\newcommand\N{\mathbb{N}}
\newcommand\Z{\mathbb{Z}}
\newcommand\T{\mathbb{T}}

\newcommand\E{\mathbb{E}}

\newcommand\sX{\mathsf{X}}
\newcommand\2{\overline{2}}
\newcommand\3{\overline{3}}
\newcommand\4{\overline{4}}

\newtheorem{theorem}{Theorem}[section]
\newtheorem{proposition}[theorem]{Proposition}
\newtheorem{corollary}[theorem]{Corollary}
\newtheorem{lemma}[theorem]{Lemma}
\theoremstyle{definition}
\newtheorem{definition}[theorem]{Definition}
\newtheorem{example}[theorem]{Example}
\newtheorem{remark}[theorem]{Remark}
\newtheorem{remarks}[theorem]{Remarks}

\begin{document}

\title[Directional expansiveness]{Directional expansiveness for $\R^d$-actions and for Penrose tilings}

\author{Hyeeun Jang}
\address{Department of Mathematics, Abilene Christain University, Abilene, TX 79601}
\email{hxj22b@acu.edu}
\author{E. Arthur Robinson, Jr.}
\address{Department of Mathematics, George Washington University, Washington, DC 20052}
\email{robinson@gwu.edu}
\date{June 9, 2025, (Version 2)}

\begin{abstract}
We define and study two kinds of directional expansiveness, weak and strong,
for an action $T$ of $\R^d$ on a compact metric space $X$.
We show that for $\R^2$ finite local complexity (FLC) tiling dynamical systems,
weak and strong expansiveness are the same, and are both equivalent to a simple coding property.  
Then we show for the Penrose tiling dynamical system, which is FLC, there 
are exactly five non  expansive directions,  the directions 
perpendicular to the 5th roots of unity. We also study Raphael Robinson's set of $24$ Penrose Wang tiles and show the corresponding Penrose Wang tile dynamical system is strictly ergodic. Finally, we study two
deformations of the Penrose Wang tile system, 
one where the square Wang tiles are all deformed into a $2\pi/5$ rhombus, and another where they are deformed into a set of eleven tetragon tiles. We show both of these are topologically conjugate to the 
Penrose tiliing dynamical system. 
\end{abstract}


\maketitle

\section{Introduction}

Let  $X$ be a compact metric space with metric $\rho$. 
A homeomorphism 
$T$ of $X$ 
is called {\em expansive} (see  \cite{Wal}) if there exists a 
$\delta>0$  
so that  $\rho(T^n(x),T^n(y))<\delta$ for all $n\in\N$ implies $x=y$. 
Expansiveness means 
two distinct points cannot always stay 
close. 
Examples of expansive homeomorphisms include Anosov diffeomorphisms, Axiom A diffeomorphisms 
restricted to their non-wandering sets, and finite alphabet \emph{subshifts}. 
It was shown in  \cite{Reddy} and \cite{KR} that every expansive homeomorphism is a factor 
of a finite alphabet subshift, although the converse is false \cite{Down}. 

The definition of expansiveness readily generalizes to continuous $\Z^d$-actions $T$
on compact metric spaces
$X$. However, when $d>1$,  $\Z^d$-actions $T$ offer additional possibilities. 
Following in the footsteps of 
Milnor's theory of directional 
entropy (see \cite{Milnor}, \cite{Milnor1}), 
Boyle and Lind \cite{BL} defined and studied directional expansiveness for 
$\Z^d$-actions. In this theory, an $e$-dimensional subspace $\cV\subseteq\R^d$,$e<d$,
 is called an {\em $e$-dimensional direction}.
Boyle and Lind \cite{BL} showed that for any $\Z^d$ action $T$, 
the set  $\E_e(T)$ of $e$-dimensional  expansive 
directions is open in the Grassmanian manifold ${\mathbb G}_{d,e}$,
and the set $\N_e(T)$ of non-expansive directions is always non-empty.
For every possible proper open set $\cO\subseteq{\mathbb G}_{2,1}$ with $\#(\cO^c)>1$,
 except for the case  $\#(\cO^c)=1$,  
 Boyle and Lind \cite{BL} gave examples of $\Z^2$ actions $T$ with set $\E_1(T)=\cO$.
Hochman \cite{Hoch} later provided examples with $\#(\N_1(T))=1$.  

A notion of expansiveness for flows $T$
(i.e., $\R$ actions) on compact metric spaces $X$ was introduced by Bowen 
and Walters in \cite{BW}. A flow $T$ is \emph{expansive} if for all $\epsilon>0$, there is a
$\delta>0$, so that for any $x,y\in X$ and any homeomorphism $h:\R\to\R$ with $h(0)=0$, 
 $\rho(T^t x,T^{h(t)} y)<\delta$ for all $t\in\R$ implies 
$y=T^s x$ for some $|s|<\epsilon$.
They proved that any fixed point free expansive flow on a 1-dimensional space $X$ 
is a quotient of a 
continuous suspension of a finite alphabet subshift. 
In the context of tiling dynamical systems $T$, Frank and Sadun \cite{FS} studied two 
types of expansiveness for $\R^d$-actions, which we call here 
\emph{weak}  and \emph{strong} expansiveness.
 Strong expansiveness implies weak expansiveness,  but not conversely in general. However,  
Frank and Sadun \cite{FS} showed that the converse holds 
for tiling dynamical systems satisfying the finite local complexity (FLC) property,
(Definition~\ref{flc}). 

In this paper, we define directional versions of 
both weak and strong expansiveness for
an $\R^d$ action $T$ on a compact metric space $X$. 
The definition, which requires some care, is given in Section~\ref{derd}. 
And while weak directional expansiveness does not generally imply strong directional expansiveness,
we show they are 
equivalent for finite local complexity (FLC)  tiling 
dynamical systems. Also, in the FLC  case, directional expansiveness is equivalent to a natural coding property.

Once the basic theory of directional expansiveness for $\R^d$-actions has been established, we move to our main study:  
describing the expansive directions for the Penrose tiling dynamical system. 
Our main theorem is the following:

\begin{theorem}\label{mainthm}
For the Penrose tiling  dynamical system $(\sX,\R^2,T)$, (which is FLC),   
$($see Section~{\rm \ref{pt}}$)$, 
there are exactly five $1$-dimensional 
non-expansive directions:
the five directions perpendicular to $5^{th}$ roots of unity. 
\end{theorem}

Our  proof uses the translation equivariant version of 
 de Bruijn's \cite{DB} duality theory for pentagrid tilings, used in \cite{Robinson1}
to prove that the Penrose tiling dynamical system is 
an almost 1:1 extension of a Kronecker $\R^2$-action on the 4-torus $\T^4$ (stated here as Theorem~\ref{RPenrose}). 
Our proof of Theorem~\ref{mainthm} takes a closer look at factor map $\varphi$ in the proof of   
Theorem~\ref{RPenrose} in \cite{Robinson1}.

We also study Rapheel Robinson's  set of $24$ aperiodic Penrose Wang tiles (see \cite{GS}) and some related 
dynamical systems. In particular, we show the  $\Z^2$ subshift of finite type $S$ corresponding to the Penrose Wang tiles 
is minimal, uniquely ergodic with pure point spectrum.
We compute its point spectrum group and 
expansive directions. 
\vskip -.3in
\subsection{Conventions}
We will identify a $d$-tuple ${\bw}=(w_0,w_1,\cdots, w_{d-1})\in\mathbb{R}^d$ 
with a $d\times 1$ column vector $[w_0 w_1 \cdots w_{d-1}]^t$. The usual dot product is denoted by $\bw\cdot\bv$.
The vectors $\be_0,\be_1,\dots,\be_{d-1}$ denote the usual standard basis for $\R^d$.
All norms are Euclidean, 
and for $d\ge 2$, 
$B_r(\bw):=\{\bv\in\R^d: ||\bv-\bw||<r\}$ denotes the open radius $r$ ball (or disc in the case $d=2$)
in $\R^d$.
For $d> 1$,  $\T^d:=\R^d/\Z^d$, with  $\T:=\R/\Z$ and $S^1=\{z\in\C:|z|=1\}$. We often identify $\T^d$ with $[0,1)^d$.
For $\bw\in\R^d$ we write $\lfloor \bw\rfloor\in\Z^d$ for the coordinate-wise floor,
and define $\{\bw\}:= \bw- \lfloor \bw\rfloor\in[0,1)^d$.
Throughout, we let $\gamma:=\frac{1+\sqrt{5}}{2}\approx 1.6180\dots$ and $\alpha:=1/\gamma=\gamma-1\approx 0.6180\dots$.

\subsection{Acknowledgements}
Most of the results in this paper are from the first author's 2021 Ph.D. dissertation \cite{Jang}
at the George Washington University. Theorem~\ref{mainthm} answers a question Mike Boyle 
once asked the second author. The authors thank S\'ebastien Labb\'e for pointing out that 
Figure~\ref{markov} was mislabeled in an earlier version.

\section{Expansiveness and directional expansiveness}\label{derd}

\subsection{Group actions} \label{ga}
Let $X$ be a compact metric space with metric $\rho$, 
and let $G$ be a locally compact complete metric abelian  group.  
In this paper, the non-compact group $G$  
will always be  $\Z^d$, $\R^d$, or a 
subgroup of one of these, and the metric on $G$ will always be 
the Euclidean norm $||\cdot||$. 

A {\em $G$-action} $T$  
on $X$ is a continuous mapping 
$T:G\times X \to X$, written $T^\bg x:=T(\bg,x)$, 
 satisfying $T^\bZr={\rm id}$ and $T^{\bg+\bh}=T^\bg\circ T^\bh$ for all $\bg,\bh\in G$. 
We  assume the acting group $G$ is non-compact and
think of a $G$-action as a \emph{dynamical system}, denoted $(X,G,T)$ or sometimes $T$ for short.
A $G$-action $T$  is 
\emph{free}, if  every $x\in X$ is  \emph{aperiodic}:
$T^\bg x=x$ 
implies $\bg={\bZr}$, and  \emph{locally free} if there exists a $\omega>0$ so that 
$T^\bg x=x$ implies $||\bg||> \omega$. 

A dynamical system $(X,G,T)$ is {\em minimal} if  any closed 
$T$-invariant subset $K\subseteq X$ satisfies either $K$ or $K^c$ is empty, or 
equivalently, $X=\overline{O_T(x)}$ for some $x\in X$, where $O_T(x):=
\{T^\bg x:\bg\in G\}$ is the \emph{orbit} of $x$ under $T$,
(see \cite{cdevries}). 
We say
 $(X,G,T)$ is 
 {\em uniquely
ergodic} if there is a unique $T$-invariant Borel probability measure $\mu$ on $X$. 
\emph{Strictly ergodic} means both minimal and uniquely ergodic, either of 
which implies $T$ is free. 

\begin{definition}\label{equivalences}
For $(X,G,T)$ and $(Y,G,K)$, we
say $K$ is a \emph{factor} of $T$ 
(or $T$ is an \emph{extension} of $K$) if there is 
a continuous surjection $\varphi:X\to Y$, the \emph{factor map}, 
such that $K^\bg( \varphi(x))=\varphi(T^\bg x)$ for all $\bg\in G$.
An extension is {\em almost $1\!\!:\!\!1$} 
if  $Y_1:=\{x\in Y:\#(\varphi^{-1}(x))=1\}$ is 
dense $G_\delta$. 
In the case $\varphi$ is a homeomorphism,  $K$ and $T$ are said to be \emph{topologically conjugate}.
We say $K$ and $T$ \emph{continuously orbit equivalent}\footnote{This is called ``conjugacy
of flows'' in \cite{BW} but we avoid that terminology as it is too close to ``topological conjugacy''. Sometimes this is called ``flow equivalent''.} if there is a 
homeomorphism $\psi:X\to Y$ so that 
$\psi(O_{T}(x))=O_{K}(\psi(x))$. 
\end{definition}

\subsection{Kronecker dynamical systems}
(See e.g., \cite{Robinson3}, \cite{Wal}) The \emph{dual group}  $\widehat G$ 
of a locally compact metric abelian group $G$ (see e.g., \cite{Wal}) is the locally compact metric abelian group 
of continuous homomorphisms ${\chi}:G\to S^1$ (in the compact-open topology).
In particular, $\widehat{\R^d}=\R^d$ with 
$\chi_\bh(\bt)=e^{2\pi i (\bh\cdot\bt)}$, and 
$\widehat{\Z^d}=\T^d$ 
with $\chi_\bh(\bn)=e^{2\pi i (\bh\cdot\bn)}$, both uncountable.
But the compact group $\T^d$  has a countable discrete dual $\widehat{\T^d}=\Z^d$
with  $\chi_\bn(\bh)=e^{2\pi i (\bh\cdot\bn)}$.

An \emph{eigenfunction} $f$ for a dynamical system $T$, with \emph{eigenvalue}
$\chi\in \widehat G$, is a continuous\footnote{
It is also possible for some $T$ to have discontinuous eigenfunction, but for the examples we consider here,
this is not the case.
}
 complex valued function satisfying 
$f(K^\bg y)=\chi(\bg)f(y)$.  
The set of eigenvalues $\Sigma_T\subseteq\widehat G$ is called the \emph{point spectrum} of $T$.
If $T$ is minimal, every eigenvalue $f$ is simple, $|f|=\text{const.}$, and $\Sigma_T\subseteq\widehat G$ is a countable discrete subgroup that 
is a topological conjugacy  invariant.

For an infinite compact metric abelian 
 group $Y$, let $i:G\to Y$ be a continuous injective homomorphism. 
We define an \emph{algebraic Kronecker} dynamical system 
$(Y,G,K)$ to be the free $G$-action $K^\bg y:=y+i(\bg)$. It satisfies
$\widehat K\cong \Sigma_K$ via $\chi\mapsto \chi\circ i$.
Minimality is equivalent to  
$\overline{i(G)}=K$, or  $\chi\mapsto \chi\circ i$ injective, (see \cite{Robinson3}). More generally, 
we call $(X,G,T)$ \emph{Kronecker} if it is topologically conjugate
to an algebraic Kronecker dynamical system. For a free Kronecker system, the
Halmos-von Neumann Theorem, (see e.g., \cite{HK}) says  $\Sigma_K$ is a complete topological conjugacy invariant. 
Also, every countable dense subgroup $\widehat Y\subseteq \widehat G$ occurs 
as $\Sigma_K$ for some $K$.

Now consider $(X,G,T)$ and a Kronecker system $(Y,G,K)$ so that 
$\Sigma_T=\Sigma_K$. Then $(Y,G,K)$ is a  factor of $(X,G,T)$,
called its \emph{Kronecker} (or  \emph{maximal equicontinuous}) factor. If the factor map $\varphi$ 
is almost 1:1 then we call $(X,G,T)$ \emph{almost automorphic}, and call
$\mathfrak{M}_T:=\{\#(\varphi^{-1}(y)):y\in Y\}\subseteq\N\cup\{+\infty\}$, 
the \emph{preimage multiplicities} (called the \emph{thickness spectrum} in \cite{Robinson1}, and 
called \emph{set of fiber cardinalities} in \cite{sebastian}).

\subsection{$\Z^2$ subshifts}\label{subshifts}

For finite  \emph{alphabet} $\cA=\{0,1,\dots,n-1\}$, $n\ge 2$, we call  
the product topology compact set $\cA^{\Z^2}= 
\{z=(z_\bm)_{\bm\in \Z^2}:z_\bm\in\cA\}$ 
the $\Z^2$ \emph{full shift} space,  
equipped with the 
$\Z^2$ \emph{shift} action $S$ 
defined by
$(S^\bn z)_\bm=z_{\bm-\bn}$. 
The $\Z^2$ \emph{full shift} dynamical system is $(\cA^{\Z^2},\Z^2,S)$, and a $\Z^2$ \emph{subshift} is 
$(Z,\Z^2,S)$ where $Z\subseteq \cA^{\Z^2}$ is any closed and $S$-invariant.

Given a rectangle $B\subseteq \Z^2$, $\#(B)>1$,   
a \emph{word} on $B$ is $b=(b_\bm)_{\bm\in B}\in \cA^B$. 
We say the word $b=(b_\bm)_{\bm\in B}$ 
occurs in $z\in\cA^{\Z^2}$ 
if $(z_{\bm-\bn_0})_{\bm\in B}=(b_\bm)_{\bm\in B}$ for some $\bn_0\in \Z^2$.
We call a subshift $Z$ 
a \emph{subshift of finite type} (SFT) if there is a \emph{finite} set $\cF$ of 
\emph{forbidden words}, so that  
now word in $\cF$ occurs in any $z\in Z$.
For a special case, 
called a $1$-step SFT,
we let  $\cF=\cF_0\cup\cF_1$ where 
$\cF_j\subseteq\cA^{\{{\bZr},\be_j \}}$.
A $1$-step SFT can also be described in terms of 
\emph{permitted words} $\cP_j=\cA^{\{{\bZr},\be_j \}}\backslash\cF_j$. 

Let $(Z,\Z,S)$ be a $\Z$ subshift, 
and let 
$\cA\otimes\cA:=\{a\otimes b:a,b\in \cA\}$.
We define the \emph{tensor product}  of $z,z' \in Z$ by 
 $z\otimes z\,':=(z_{n_0}\otimes z'_{n_1})_{(n_0,n_1)\in\mathbb Z^2}\in (\cA\otimes\cA)^{\Z^2}$, and
 define the \emph{tensor square} 
  $Z\otimes Z\subseteq (\cA\otimes\cA)^{\Z^2}$ 
 to be the set of all tensor product sequences $z\otimes z'$ for $z,z'\in Z$. The $\Z^2$ \emph{shift} 
 map on $Z\otimes Z$ is given by 
 $(S\otimes S)^{(k_0,k_1)}(z_{n_0}\otimes z_{n_1})_{(n_0,n_1)\in\mathbb Z^2}=(z_{n_0-k_0}\otimes z_{n_1-k_1})_{(n_0,n_1)\in\mathbb Z^2}$.
The tensor square dynamical system $(Z\otimes Z, \Z^2,S\otimes S)$
is a $\Z^2$ subshift,  
which is   
minimal,  uniquely ergodic, or finite type, if and only if $(Z,\Z,S)$ is the same.

\subsection{Expansiveness for $\Z^d$ actions}

A $\Z^d$ dynamical system $(X,\Z^d,T)$ is
 {\em expansive} if there is a $\delta>0$ so that 
$\rho(T^\bn x ,T^\bn y)<\delta$ for all 
$\bn\in \Z^d$ implies $x=y$. 
Expansiveness is an invariant of topological conjugacy (see \cite{BW} for the case $d=1$).

Now we want to 
 describe expansiveness for an $\R^d$ dynamical system $(X,\R^d,T)$, but simply 
replacing $\Z^d$ with $\R^d$ in the definition above
does not quite work (see \cite{BW}). Actually, we will describe two versions of expansiveness 
for $\R^d$ dynamical systems, both of which were
studied for $d=1$ by Bowen and Walters \cite{BW},
(see also Frank and Sadun \cite{FS}).
We call them \emph{weak} and \emph{strong expansiveness}. 

\begin{definition} \label{DFrankSadun}
Let $(X,\R^d,T)$ 
be a  locally free 
$\R^d$ dynamical system.
({\it i\,}) $T$
is called \emph{weakly expansive}
if for all $\epsilon>0$ there exists  
a $\delta>0$ 
so that for any $x,y\in X$, 
 $\rho(T^{\bt}x, T^{\bt}y)<\delta$ for all ${\bt}\in\R^d$ implies 
$y=T^{{\bs}_0}x$ for some $||{\bs}_0||<\epsilon$. 
({\it ii\,})  $T$ is called \emph{strongly expansive} if 
for all $\epsilon>0$  there is a $\delta>0$ 
so that 
for any $x,y\in X$ 
and any homeomorphism 
$h:\R^d\to\R^d$ with 
$h({\bZr})={\bZr}$, 
$\rho(T^{\bt}x, T^{\bh(\bt)}y)<\delta$ for all $\bt\in\R^d$
implies
$y=T^{{\bs}_0}x$ for some $||{\bs}_0||<\epsilon$.  
\end{definition}

\begin{remarks}\ 
\begin{enumerate}
\item {Both weak and strong expansiveness are invariants of topological conjugacy.}
\item {The assumption ``locally free'' is a convenience that eliminates trivial  cases  like non isolated
fixed points (see \cite{BW})}. Locally free is automatic for FLC tiling dynamical systems (see Remark~\ref{lf}).  
\item In the case $d=1$, Bowen and Walters  \cite{BW} call our 
 ``strong expansiveness'' ({\it ii\,})  ``expansiveness''. They also discuss 
our ``weak expansiveness''  ({\it i\,}), without naming it, but  dismiss it  
because it is not an 
invariant of ``conjugacy of flows'': our ``continuous orbit equivalence''. 
\item Other versions of ({\it i\,}) and ({\it ii\,})  ({\it a priori} a little stronger) were 
studied for tiling dynamical systems by Frank and Sadun \cite{FS}.
Their versions of ({\it i\,})  and ({\it ii\,}) , which are equivalent to ours and each other in the FLC case, 
have no $\epsilon$, and the conclusion is  $||\bs_0||<\delta$. They refer to ({\it i\,}) 
 as ``geometric expansiveness''.  
\item There are also versions of ({\it i\,})  and ({\it ii\,}) 
where the conclusion is simply that $y=T^{\bs_0}x$, i.e., with no condition on $||\bs_0||$. 
 Gura \cite{Gura} calls ({\it i\,})  ``separating'' for $d=1$, and shows  the horocycle flow
 satisfies ({\it i\,})  but not ({\it ii\,}) .  
For $d=1$, Katok and Hasselblatt \cite{KH} call this version of ({\it ii\,})  for $d=1$ ``expansiveness''. 
All (weak and strong) are equivalent in the FLC case. 
\end{enumerate} 
\end{remarks}

\subsection{Directional Expansiveness for $\Z^d$} 

For an $e$-dimensional direction (i.e., a subspace) $\cV\subseteq\R^d$,  $1\le e<d$, we define
$\cV^r:= \bigcup_{{\bv}\in \cV}B_r({\bv})$. 

\begin{definition}[Boyle and Lind, \cite{BL}]\label{BLexp}
Let $(X,\Z^d,T)$ be a $\Z^d$ action and  
let $\cV\subseteq\mathbb{R}^d$ be  a $1\le e<d$ dimensional direction.
Then $T$ is  {\em expansive} in the direction $\cV$ if there exist a
$\delta>0$ and an $r>0$ 
such that if $x,y\in X$ satisfy $\rho(T^\bn x,T^\bn y)\leq\delta$ 
for all $\bn\in \Z^d\cap \cV^r$, 
then $x=y$. 
\end{definition} 

In 
\cite{BL}, the sets of $e$-dimensional expansive and non-expansive directions are denoted  by $\E_d(T)$ and $\N_d(T)$.
Let $(Z,\Z^d,S)$ be a $\Z^d$ subshift, and let $J\subseteq\R^d$. We say a set $J$ 
\emph{determines} $Z$ if for $z,z'\in Z$, $z[J]=z'[J]$ implies $z=z'$. 
In the paragraph following 
(2--1) in \cite{BL}, Boyle and Lind give the following useful characterization of directional 
expansiveness for a subshift.

\begin{lemma}[Boyle and Lind, \cite{BL}]\label{asashift}
For a $\Z^d$ subshift $(Z,\Z^d,S)$ and an $e$-dimensional direction $\cV\subseteq\R^d$, $1\le e<d$, 
$\cV$ is an expansive direction if and only if $\cV^r$ determines $Z$ for some $r>0$.
\end{lemma}

\begin{lemma}\label{recover}
Let $(Z\otimes Z,\Z^2,S\otimes S)$ be the tensor square 
of a $\Z$-subshift. For $J\subseteq\Z^2\subseteq\R^2$, 
suppose  $\pi_0(J)=\Z$ and $\pi_1(J)=\Z$, where $\pi_0,\pi_1:\Z^2\to\Z$ 
are the coordinate projections. 
Then $J$ determines $Z\otimes Z$.
\end{lemma}

\begin{proof}
To find $z_{n_0}\otimes z'_{n_1}$ for arbitrary $(n_0,n_1)\in\Z^2$ we use 
$\pi_0(J)=\Z$ and $\pi_1(J)=\Z$
to find $(n_0,m_1),(m_0,n_1)\in J$. Then we use the first factor of $z_{n_0}\otimes z'_{m_1}$, and 
the second factor of $z_{m_0}\otimes z'_{n_1}$ to make $z_{n_0}\otimes z'_{n_1}$.
\end{proof}

\begin{corollary}
For any tensor square 
subshift $(Z\otimes Z,\Z^2,S\otimes S)$, the non-expansive directions  
$\N_1(S\otimes S)$
consist of exactly the horizontal and vertical directions.
\end{corollary}

\begin{proof}
Let $\cV$ be horizontal and let $r>0$. Find $z', z''\in Z$ so that 
$z'_j=z''_j$ for $|j|\le r$ but $z'_k\ne z''_k$ for some $|k|>r$. Then for any $z\in Z$, 
$(z\otimes z')[\cV^r]=(z\otimes z'')[\cV^r]$, 
since $z_i\otimes z'_j=z_i\otimes z''_j$ for $i\in \Z$ and $|j|\le r$, i.e., $(i,j)\in\cV^r$.
But $(z\otimes z')\ne  (z\otimes z'')$, since $z'_k\ne z''_k$ implies $z_i\otimes z'_k\ne z_i\otimes z''_k$.
Thus  $\cV^r$ does not determine $z\otimes z'$ for any $r>0$.
The same argument works for the vertical direction.

Now suppose $\cV$ is neither vertical nor horizontal. Define 
$Q_{(n_0,n_1)}:=[n_0-\frac12,n_0+\frac12)\times [n_1-\frac12,n_1+\frac12)$.
Choose $r>0$ large enough that 
$J:=\bigcup_{\{\bn\in\Z^2:Q_\bn\cap \cV\ne\emptyset\}}Q_\bn\subseteq \cV^r$.
$\pi_0(J)=\Z$ and $\pi_1(J)=\Z$, so the same holds for $\cV^r$.
Lemma~\ref{recover} shows $\cV^r$ determines $Z\otimes Z$, 
and the result follows from Lemma~\ref{asashift}.
\end{proof}

\subsection{Directional Expansiveness for $\R^d$} 
It turns out that the correct definition of direction $\cV$ expansiveness for an $\R^d$ action $T$ is not 
simply the definition of 
expansiveness for $T$, restricted to $\cV$, even though such a definition does work in some directional theories.
The restriction of $T$ to $\cV$,  defined by $(T|_\cV)^\bt:=T^\bt$ for $\bt \in \cV$,
is a $\cV$ action, but can be made into an $\R^e$ action
by a choice of a basis (see e.g., \cite{BL}). This does not affect questions of expansiveness. 

\begin{proposition} 
Let $(X,\R^d,T)$ be a  locally free $\R^d$ action. 
Then for any  direction $\cV$, 
the restriction $(X,\cV,T|_\cV)$ to $\cV$ is not 
weakly or strongly expansive. 
\end{proposition}  

\begin{proof}
Suppose $T|_\cV$ is weakly expansive and choose 
$\epsilon<\omega/2$ where $\omega$ is the constant for locally free.
Let $\bp\in\R^d\backslash \cV$ with $||\bp||=1$, and define 
a continuous
real-valued function
on $[0,1]\times X$ by $f(x,r):=\rho(x,T^{r\bp}x)$. Then 
$f(0,x)=0$, and  $f(r,x)>0$ for $r<\omega/2$. For $\delta>0$, 
find $0<r<\omega/2$ so that 
$f(r,x)<\delta$ for all $x\in X$. 
Fix $x\in X$ and let $y=T^{r\bp}x$. 
Then for $\bt\in\cV$,  
$\rho(T^{\bt}x, T^{\bt}y)=\rho(T^{\bt}x, T^{r\bp}T^\bt x)=f(T^{\bt}x,r)<\delta$.
Now, $T|_\cV$ weakly expansive would mean $y=T^{r\bp}x=T^{\bs_0}x$ for 
some $\bs_0\in\cV$ with $||\bs_0||<\epsilon<\omega/2$. But 
then $T^{\bs_0-r\bp}x=x$. Since $|| \bs_0-r\bp||\le ||\bs_0||+r||\bp||<\omega$,
locally free implies $\bs_0=r\bp$, contradicting 
$r\bp\notin \cV$.
\end{proof}

\begin{definition}\label{directional}
Let $(X,\R^d,T)$ be a  locally free $\R^d$ action, and
let $\cV\subseteq \R^d$ be an $e$-dimensional direction, $1\le e<d$. 
({\it i}) 
$T$ is called \emph{weakly {\rm (}directionally{\rm )} expansive} in the direction $\cV$ if for all $\epsilon>0$
there is a 
$\delta>0$ such that for all $x,y\in X$, 
$\rho(T^{\bt}x, T^{\bt}y)<\delta$, for all ${\bt}\in \cV$ implies 
$y=T^{\bs_0}x$ for some $\bs_0\in \R^d$ with 
$||{\bs}_0||<\epsilon$. 
({\it ii}) 
$T$ is called \emph{strongly {\rm (}directionally{\rm )} expansive} in the direction $\cV$ if for all $\epsilon>0$
there is a  $\delta>0$ such that 
for any $x,y\in X$ and any  
$h:\cV\to \cV$ with $h({\bZr})={\bZr}$,  
 $\rho(T^{\bt}x, T^{h(\bt)}y)<\delta$ for all $\bt\in \cV$ implies $y=T^{\bs_0}x$ for some $\bs_0\in \R^d$ 
with $||{\bs}_0||<\epsilon$. 
\end{definition}

\section{Tiling dynamical systems}\label{tilings}

In this section we define finite local complexity (FLC)  $\R^2$ tiling spaces and $\R^2$ tiling dynamical systems,
and describe their basic properties (see e.g., \cite{Robinson} or \cite{BG}). 
Similar considerations hold for $\R^d$, $d>2$, but for simplicity we do not discuss them here.

\subsection{Basic definitions}\label{FLCtilings}

A {\em tile} is a subset $D\subseteq\R^2$ homeomorphic to a closed disk, with a Lebesgue measure 
zero boundary (e.g., a convex polygon).   
A {\em tiling} $x$ of $\R^2$ is a set of the tiles that cover $\R^2$, intersecting only on their boundaries. 
A {\em tiling patch} $x$ in $\R^2$ is a finite collection 
of the tiles whose union is a closed topological disc, and a 
{\em tiling swatch} $x$ is a collection of tiles whose union is closed and simply connected.

If $x$ is a tiling of $\R^2$, 
and $S\subseteq\R^2$, we let $x[S]$ be the smallest patch or swatch, the union of whose tiles contains $S$. 
Two tiles, patches, swatches or tilings are said to be \emph{equivalent} if they are congruent via a translation.
Sometimes geometrically equivalent tiles are ``decorated'' with colors or markings that make them inequivalent. We denote the translation of a tiling (or tile or patch or swatch) $x$ by ${\bt}\in\R^2$ by $T^{\bt}x:=\{D-{\bt}:D\in x\}$.
We call a set $p$ of inequivalent tiles
 in $\R^2$ a set of {\em prototiles}. 
A tiling (or path or swatch) $x$ \emph{by} $p$ means $x$ each tile in $x$ is equivalent to a tile in $p$.
We want to look at sets of all
tilings by $p$, but without further 
restrictions, such a  set is often too large. 

\begin{definition}\label{flc}
Let $p$ be a finite set of prototiles in $\R^2$. A \emph{local rule} is a finite set $p^{(2)}$ of 
inequivalent  $2$-tile patches by $p$.
We let  $X_p$ (or $p^*$) denote the set of all tilings (or patches) by $p$ such that every 
$2$-tile patch in any $x\in X_p$ (or any $y\in p^*$) is equivalent to a patch in  
$p^{(2)}$. We call $X_p$ the {\em finite local complexity} 
(abbreviated FLC)  {\em full tiling space} defined by $p$, (with  the
local rule $p^{(2)}$ being implicit). 
\end{definition}

We will always assume that $p$ and $p^{(2)}$ are such that 
$X_p\not=\emptyset$. This is a nontrivial assumption in general
(see Section~\ref{swangtiles}),
but is easy to ascertain in some cases
(see Example~\ref{rhombic}), 
or is known in some others (see Remark~\ref{defx}). 

There is a well known compact metric topology on any FLC tiling spaces $X_p$ 
in which two tilings are close if, after a small translation, they agree on a large ball around the origin. 
Here is one version of such a metric $\rho$  (see \cite{Robinson}, \cite{sadun}).

\begin{definition}\label{tmetric}
For each tile $D\in p$, let $\delta(D)=\inf\{||\br||:T^\br D\cap D=\emptyset\}$
and let $\Delta_p:=\min_{D\in p}\delta(D)$.
For $x,y\in X_p$,  $x\ne y$, let
$R(x,y):=\sup\left\{r: (T^{\br_x}x)[B_{r}({\bZr})]=(T^{\br_y}y)[B_{r}({\bZr})]\text{ some } 
||\br_x||,||\br_y||<1/r\right\}$,
Define the \emph{tiling metric} $\rho(x,y):=\min(\Delta_p,1/R(x,y))$.
\end{definition}

A well known result of Rudolph  \cite{Rudolph} is that an FLC 
full tiling space $X_{p}$ is a compact metric space (see also \cite{Robinson}, \cite{sadun}).
The following facts will be useful.

\begin{lemma}\label{useful}
Let $X_p$ be an FLC tiling space, and let $x,y\in X_p$.
\begin{enumerate}
\item\label{useful3}
If $(T^\br x)[\{\bt\}]=(T^{\bs}x)[\{\bt\}]$ with $||\br-\bs||<\Delta_p$ then $\br=\bs$.
\item\label{useful1}
If $\rho(x,y)<\delta<\min(1/3,\Delta_p)$ and $r=1/\delta$,
then there exists a $\bt\in\R^d$ with $||\bt||<\delta$ so that 
$x[B_{(8/9)r}({\bZr})]=(T^{\bt} y)[B_{(8/9)r}({\bZr})]$.   
\item \label{useful2} 
If $\rho(T^\br x, T^\bs x)<\Delta_p$  or $||\br-\bs||<\Delta_p$
then $\rho(T^\br x, T^\bs x)=||\br-\bs||$. 

\end{enumerate}
\end{lemma}

\begin{proof}\
\begin{enumerate}
\item  
Suppose $x[\{\bt\}]=(T^{\br}x)[\{\bt\}]$ for $||\br||<\Delta_p$. Then for some $D\in x[\{\bt\}]$ we also 
have $T^\br D\in x[\{\bt\}]$, with $T^\br D\cap D \ne\emptyset$. So $T^\br D=D$ and $\br={\bZr}$.
\item If $0<\rho(x,y)=\delta=:1/r<\Delta_p$ then $(T^{\br}x)[B_{r}({\bZr})]=(T^{\bs}y)[B_{r}({\bZr})]$
for some $||\br ||,||\bs ||<\delta/2$. So $x[B_{r}({\br})]=(T^{\bs-\br}y)[B_{r}({\br})]$ and 
$\bt:=\bs-\br$ satisfies $||\bt||<\delta$. It follows that $x[B_{(9/10)r}({\bZr})]=(T^{\bt}y)[B_{(9/10)r}({\bZr})]$,
because $1/r=\delta<1/3$ implies $B_{(9/10)r}({\bZr})\subseteq B_{r}({\br})$.
\item It suffices to prove $\rho(x,T^\br x)=||\br||$. Clearly $\rho(x,T^\br x)\le ||\br||$. 
Now suppose $\rho(x,T^\br x)=\delta=1/r<||\br||$. By (2) there is $||\bs||<\delta$ with
$(T^\bs x)[B_{(8/9)r}({\bZr})]=(T^{\br}x)[B_{(8/9)r}({\bZr})]$,
so by (1), $\br=\bs$, contradicting $\delta=||\bs||<||\br||$.  
The converse is clear.
\end{enumerate}
\vskip -.2in
\end{proof}

\begin{definition}
For a full FLC tiling space $X_p$,
we call any closed $T$-invariant subset $X\subseteq X_{p}$ 
an \emph{FLC tiling space} (i.e., subspace)
and call $(X,\R^2,T)$ an FLC \emph{tiling dynamical system}. 
\end{definition}

Let $X$ and $X'$ be tiling spaces. For $r>0$, consider the sets of patches $X_{\{r\}}:=\{x[B_r(\bZr)]:x\in X\}$ and $X'_{\{0\}}:=\{x'[\{\bZr\}]:x'\in X'\}$.

\begin{definition}\label{local}
A map $C:X\to X'$ is called \emph{local} with \emph{local radius} $r$
(see \cite{Robinson}, \cite{BG}) if there is a map $c: X_{\{r\}}\to X'_{\{0\}}$ so that  
$C(x)[\{\bZr\}]=c(x[B_r(\bZr)])$ for all $x\in X$. 
\end{definition}

Note that a local map is uniformly continuous, and satisfies  
$C(T^\bt x)=T^\bt C(x)$ for  $x\in X$ and $\bt\in\R^2$. A local map is a factor map if it is onto.  
If a local homeomorphism $C:X\to X'$ 
has a local inverse, we say $X$ and $X'$ are  \emph{mutually locally derivable} (MLD),
and say $C$ is a MLD topological conjugacy. 
Local maps between FLC 
tiling spaces
are analogous to sliding block codes between $\Z^d$ subshifts. In the subshift case, 
the Curtis/Lyndon/Hedlund Theorem says every topological conjugacy is a sliding block code. 
However, for FLC tiling dynamical systems there are topological conjugacies that 
are not MLD (see \cite{petersen}, \cite{holton}, \cite{BG}, and Theorem~\ref{pwt}).

\begin{example}\label{rhombic}
For $j=0,1,\dots,4$, let ${\bv}_j:=({\rm Re}(e^{2\pi i j/5}),{\rm Im}(e^{2\pi i j/5}))$, 
the $j$th $5$th root of unity, and let  ${\bv}'_j:=(2/5){\bv}_j$.
We denote the prototile set by $p_5:=\{{\bv}'_i\wedge {\bv}'_j:0\le i<j\le 4 \}$, where 
${\bv}'_i\wedge {\bv}'_j:=\{t_0{\bv}'_i+t_1{\bv}'_j : t_0, t_1 \in[0,1]\}$.
Then $p_5$ consists  of $10$ rhombic tiles with edge length $2/5$, five with acute angle $2\pi/5$ and  
five with acute angle $2\pi/10$.
Figure~\ref{unmarked} (a) shows  ${\bv}'_0\wedge {\bv}'_4$ and ${\bv}'_0\wedge {\bv}'_3$.
All other tiles in $p_5$ are rotations of one of these.
We take local rule 
  $p^{(2)}_5$ to be the requirement that all tiles meet edge to edge. 
The corresponding 
FLC full tiling space satisfies $X_{p_5}\ne\emptyset$  since it contains 
periodic tilings (see Figure~\ref{unmarked}, (b)). 
\end{example}

\begin{remark}\label{lf}
The existence of periodic tilings $x\in X_{p_5}$ means 
that the FLC tiling dynamical system $(X_{p_5},\R^2,T)$ is not minimal, uniquely ergodic, or even free. However, 
every FLC tiling dynamical system is locally free by (\ref{useful3}) of Lemma~\ref{useful}.
\end{remark}

\begin{figure}
\includegraphics[width=9cm]{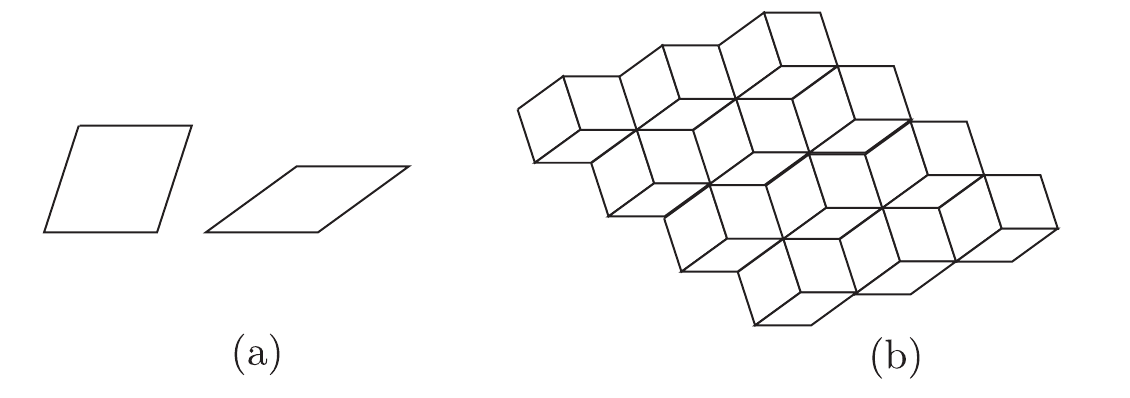}
\caption{(a) Two unmarked rhombic tiles, \emph{thick}: $\bv'_0\wedge\bv'_1$ 
and \emph{thin} $\bv'_0\wedge\bv'_3$. (b) A swatch of a periodic tiling 
$x\in X_{p_5}$, showing $X_{p_5}\ne\emptyset$, but also that $p_5$ is not an aperiodic prototile set.  
\label{unmarked}}
\end{figure}

\subsection{Expansiveness and directional expansiveness for tiling dynamical systems}\label{expan}

\begin{theorem}[Frank-Sadun, \cite{FS}]\label{wis}
For $(X,\R^2,T)$ an FLC tiling dynamical system, weak expansiveness is equivalent to strong expansiveness (we just say ``expansive'' in the FLC case).
\end{theorem}

\begin{proposition}\label{strongly}
Any  FLC tiling dynamical system $(X,\R^2, T)$ is expansive.
\end{proposition}

\begin{proof}
Given $\epsilon>0$, let $\delta<\min\left(\epsilon,1/3,\Delta_p/2\right)$, $r=1/\delta$. Let $B_0=B_{8r/9}({\bZr})$. For a sequence $\bt_n\in\R^2$, $n=1,2,\dots$, 
let $B_n:=B_{8r/9}(-\bt_n)$. Choose the $\bt_n$ so that 
$\cup_{n\ge 1} B_n=\R^2$, and so that $I_n:=B_{n+1}\cap B_n\ne\emptyset$.
Now suppose  $x,y\in X_p$ satisfy $\rho(T^\bt x,T^\bt y)<\delta$ for all $\bt\in\R^2$.
Then for each $n=1,2,\dots$, $\rho(T^{\bt_n} x,T^{\bt_n} y)<\delta$, so that by (\ref{useful1}) of Lemma~\ref{useful},
there exists $||\br_n||<\delta$ 
so that $(T^{\bt_n}x)[B_0]=(T^{\bt_n+\br_n}y)[B_0]$. Equivalently,
$x[B_n]=(T^{\br_n}y)[B_n]$. 
We \emph{claim} $\br_n=\br_1$ for all $n\in\N$.
Thus $x[B_n]=(T^{\br_1}y)[B_n]$ for all $n$, and 
since the $\R^2=\cup_{n\ge 1} B_n$,  $y=T^{\br_1}x$.
By (\ref{useful2}) of Lemma~\ref{useful}, $||\br_1||=\rho(x,y)<\delta<\epsilon$.
To prove the \emph{claim}, we prove $\br_{n+1}=\br_n$ for all $n$. Let $C_n:=B_n\cap B_{n+1}$.
Then $(T^{\br_n}y)[C_n]=x[C_n]=(T^{\br_{n+1}}y)[C_n]$. Since $||\br_{n+1}-\br_n||<2\delta<\Delta_p$, 
(\ref{useful3}) of Lemma~\ref{useful} implies
$\br_{n+1}=\br_n$. 
\end{proof}

\bigskip

We now discuss directional expansiveness for FLC tiling dynamical systems $(X,\R^2, T)$. 
Any $1$-dimensional direction $\cV$ in $\R^2$ is a line. We typically describe it either as the line generated by a 
vector $\bv\ne\bZr$, or by its slope $m\in \R$,  (with $m=\infty$ being a vertical line).

\begin{proposition}\label{FLCdirectioncor}
If $(X,\R^2,T)$ is an FLC tiling dynamical system, then a direction $\cV$ is weakly expansive if and only if it is strongly expansive in the direction $\cV$. 
\end{proposition}

Because of Proposition~\ref{FLCdirectioncor},
we will from now on 
simply refer to \emph{expansive} and \emph{non-expansive} directions
when discussing 
FLC tiling dynamical systems.
Our proof of Proposition~\ref{FLCdirectioncor}, closely follows the proof of Theorem~\ref{wis} in \cite{FS}. We  start with a lemma.

\begin{lemma}\label{added}
Let $(X,\R^2,T)$, $X\subseteq X_p$, be an FLC tiling dynamical system, and let 
$0<2\delta<\min(1/3,\Delta_p)$.
For a $1$-dimensional direction $\cV\subseteq\R^2$, and  $x,y\in X$, let
$h:\cV\to\cV$ be a homeomorphism such that $h(\bZr)=\bZr$ and  
$\rho(T^\bt x,T^{h(\bt)}y)<\delta/3$ for all $\bt\in\cV$. 
Then $||h(\bt)-\bt||<\delta$ for all $\bt\in\R$.
\end{lemma}

\begin{proof}
We prove this by induction.
Let $r=\delta/3$.
Suppose there is a $\bt\in\cV$ with $||\bt||<r$ such that $||h(\bt)-\bt||\ge\delta$. By continuity we can find 
$\bt_0\in\cV$, $||\bt_0||< r$, so that $||h(\bt_0)-\bt_0||= \delta$. Now, since 
$\rho(x,y)=\rho(T^\bt x,T^{h(\bt)}y)|_{\bt=\bZr}<\delta/3$, 
part (2) of Lemma~\ref{useful} says there is $\br\in\R^2$ with $||\br||<\delta/3$ so that 
$x[B_{2r}({\bZr})]=(T^\br y)[B_{2r}({\bZr})]$. 
Since $\bt_0\in B_{2r}({\bZr})$, we have $x[\{\bt_0\}]=(T^\br y)[\{\bt_0\}]$. Again, using  part (2) of Lemma~\ref{useful}, 
$\rho(T^{\bt_0} x,T^{h(\bt_0)}y)<\delta/3$ implies $(T^{\bt_0} x)[B_{2r}({\bZr})]=(T^{h(\bt_0)+\bs}y)[B_{2r}({\bZr})]$ 
for some $||\bs||<\delta/3$. Thus
$x[B_{2r}({\bt_0})]=(T^{h(\bt_0)-\bt_0+\bs}y)[B_{2r}(\bt_0)]$, which implies
$x[\{\bt_0\}]=(T^{h(\bt_0)-\bt_0+\bs}y)[\{\bt_0\}]$.
It follows that $(T^\br y)[\{\bt_0\}]=(T^{h(\bt_0)-\bt_0+\bs}y)[\{\bt_0\}]$.
Since $||h(\bt_0)-\bt_0+\bs-\br||\le||h(\bt_0)-\bt_0||+||\br||+||\bs||<\delta+\delta/3+\delta/3=2\delta<\Delta_p$,
part (1) of Lemma~\ref{useful} implies $h(\bt_0)-\bt_0=\bs-\br$ which implies
$\delta=||h(\bt_0)-\bt_0||=||\bs-\br||<\delta/3+\delta/3<\delta$, a contradiction. This is the base of the induction.

For the induction step, we assume that,
for some $k\ge 1$, 
 no $\bt\in\cV$ with $||\bt||<rk$ has $||h(\bt)-\bt||\ge 2\delta/3$, 
and show the same for   
$\bt\in\cV$ with $||\bt||<(k+1)r$. 
Suppose, to the contrary, 
that some such $\bt=\bt_0$ 
satisfies $||h(\bt_0)-\bt_0||=2\delta/3$. 
Let $\bs_0=kt/(k+1)$ so that $\bs_0\in\cV$, $||\bs_0||<kr$ and 
$||\bt_0-\bs_0||<r$. By the assumptions,  
$\rho(T^{\bs_0}x,T^{h(\bs_0)}y)<\delta/3$, and by  induction,  
$||h(\bs_0)-\bs_0||<2\delta/3$,
so part (2) of Lemma~\ref{useful} implies
\[
\rho(T^{\bs_0}x, T^{\bs_0}y)\le \rho(T^{\bs_0}x, T^{h(\bs_0)}y)+\rho(T^{h(\bs_0)}y, T^{\bs_0}y)
<\delta/3+||h(\bs_0)-\bs_0||<\delta.
\]
Since $||\bt_0-\bs_0||<r$, the argument above shows
$(T^{\bs_0}x)[\bt_0-\bs_0]=(T^{\bs_0+\br}y)[\bt_0-\bs_0]$ for some $||\br||<\delta/3$. Equivalently,
$x[\{\bt_0\}]=(T^\br y)[\{\bt_0\}]$. But also, the assumption, $\rho(T^{\bt_0}x,T^{h(\bt_0)}y)<\delta/3$  implies
$x[\{\bt_0\}]=(T^{h(\bt_0)-\bt_0+\br'}y)[\{\bt_0\}]$ for some $||\br'||<\delta/3$. So 
$(T^\br y)[\{\bt_0\}]= (T^{h(\bt_0)-\bt_0+\br'}y)[\{\bt_0\}]$. As above, part (1) of Lemma~\ref{useful} implies 
$h(\bt_0)-\bt_0=\br'-\br$, which yields the contradiction  $2\delta/3=|| h(\bt_0)-\bt_0||=||\br'-\br||<2\delta/3$.
\end{proof}

\begin{proof}[Proof of Proposition~\ref{FLCdirectioncor}] 
Assume that $T$ is weakly expansive in the direction $\cV$. Fix $\epsilon>0$ and find 
$\delta'>0$ for weak expansiveness ((1) of Definition~\ref{directional}), satisfying $\delta'<\Delta_p$. 
Let $x,y\in X$ 
and  let $h=h_{x,y}$ be as in part (2) of Definition~\ref{directional} satisfying: (i)
$\rho(T^\bt x,T^{h(\bt)}y)<\delta'/3$. Then by part (\ref{useful2}) of Lemma~\ref{useful} and Lemma~\ref{added} 
 we have: (ii) $\rho(T^\bt y,T^{h(\bt)}y)
 =||h(\bt)-\bt||<2\delta'/3$.
 Thus $\rho(T^\bt x,T^\bt y)\le \rho(T^\bt x,T^{h(\bt)}y)+\rho(T^{h(\bt)} y,T^\bt y)<\delta'/3+2\delta'/3=\delta'$,
by (i) and (ii). It follows from weak expansiveness (part (1) of Definition~\ref{directional}) that $y=T^{\bs_0}x$ 
 for some $\bs_0\in\R^2$ with $||\bs_0||<\epsilon$. This  implies $T$ is strongly expansive (for $\delta=\delta'/3$)
in the direction $\cV$.
\end{proof}

\begin{definition}\label{exp0}
Let $\cV$ be a 1-dimensional direction. 
We say an FLC tiling dynamical system $T$ is direction $\cV$ \emph{symbolically expansive} 
if there is an $r>0$ so that 
$x[\cV^r]=y[\cV^r]$ for all $x,y\in X$ implies $x=y$.
\end{definition}

\noindent The following should be compared to Lemma~\ref{asashift}.

\begin{proposition}\label{exp}
For an FLC tiling dynamical system $(X,\R^2,T)$, direction $\cV$ symbolic expansiveness 
 is equivalent to  direction $\cV$ expansiveness.
\end{proposition}

\begin{proof}
Suppose 
Definition~\ref{exp0} holds for 
$T$ 
in the direction $\cV=\{t\bv:t\in\R\}$, $||\bv||=1$, for  $r>0$. 
Fix $\epsilon>0$, and for $r'=(2/3)\sqrt{3}r$,
let $\delta<\min(\epsilon, 1/r',\Delta_p/2)$.
Let $\bt_n=nr'\bv$, $n\in\Z$.
Then $\cV^r\subseteq\cup_{n\in\Z}B_{r'}(\bt_n)$. 
Let $C_n:=B_{r'}(\bt_n)\cap B_{r'}(\bt_{n+1})$.
Now suppose $\rho(T^\bt x,T^\bt y)<\delta$ for all $\bt\in \cV$. 
For each $n\in\Z$ there is a 
small shift, 
$||\br_n||<\delta$, such that 
$(T^{\bt_n} x)[B_{r'}({\bZr})]=(T^{\br_n} T^{\bt_n} y)[B_{r'}({\bZr})]$,
or equivalently
$x[B_{r'}({-\bt_n})]=(T^{\br_n} y)[B_{r'}({-\bt_n})]$. 
We claim $\br_n=\br_{n+1}$ for all $n\in\Z$.  
we have $(T^{\br_n} y)[C_n]=x[C_n]=(T^{\br_{n+1}} y)[C_n]$ where 
$||\br_n||,||\br_{n+1}||<\delta$. Since $||\br_{n+1}-\br_n||<\Delta_p$,
 (\ref{useful3}) of Lemma~\ref{useful} implies
$\br_{n+1}=\br_n$. So $x[B_{r'}({-\bt_n})]=(T^{\br_0} y)[B_{r'}({-\bt_n})]$
for all $n$. But the definition of $r'$ implies 
$\cV^r\subseteq\cup_{n\in\Z}B_{r'}(\bt_n)$. 
Thus $x[\cV^r]=(T^{\br_0} y)[\cV^r]$, and by Definition~\ref{exp0}, 
$x=T^{\br_0} y$ with $||\br_0||<\delta<\epsilon$.

Conversely, suppose $T$ is weakly expansive in direction $\cV$. Given $\epsilon>0$, 
let $\delta<\min(\epsilon,\Delta_p)$, and let $r>1/\delta$.
Now suppose $x[\cV^r]=y[\cV^r]$. Then $x[B_r(-\bt)]=y[B_r(-\bt)]$ for 
all $\bt\in\cV$, or equivalently, $(T^\bt x)[B_r({\bZr})]=(T^\bt y)[B_r({\bZr})]$ 
for all $\bt\in \cV$. Thus $\rho(T^\bt x,T^\bt y)=1/r<\delta$. 
By weak expansiveness in direction $\cV$,
$y=T^{\bs_0}x$ with $||\bs_0||<\epsilon<\delta<\Delta_p$.
But   $(T^{\bs_0}x)[\cV^r]=x[\cV^r]$, so by (\ref{useful3}) of Lemma~\ref{useful}, 
we have $\bs_0={\bZr}$ and $x=y$.
\end{proof}

\section{The Penrose tiling dynamical system}\label{four}

\subsection{Penrose tilings}\label{pt}

In the next two sections (4 and 5), we discuss two types of Penrose tilings: the 
\emph{arrowed rhombic Penrose tilings} and the \emph{unmarked rhombic Penrose tilings}
(see \cite{Robinson}, \cite{Robinson1}, \cite{BG}). Later, in Section 6, 
we will discuss two additional types of tilings 
whose tiling dynamical systems 
are topologically conjugate to the Penrose tiling dynamical system, (see Theorem~\ref{pwt}).

The arrowed rhombic Penrose tiles, denoted by $p_{5{\rm a}}$, are versions of the tiles from $p_5$
(Example~\ref{rhombic}) 
marked with arrows. 
Two of these tiles  are shown 
in Figure~\ref{ftiles}(a). The entire set $p_{5{\rm a}}$
of $20$ prototiles  is obtained by taking all rotations of these two tiles by multiples of 
$2\pi/10$. To make an FLC prototile set,
the local rule $p^{(2)}_{5{\rm a}}$  requires 
two tiles to meet edge to edge with matching arrows (see Figure~\ref{ftiles}). 
\begin{figure}[h!]
\includegraphics[scale =.9]{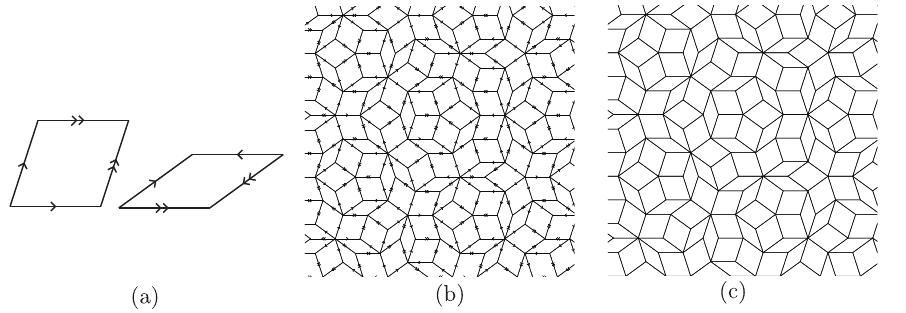}
\caption{(a) two arrowed rhombic Penrose tiles, (b) arrowed rhombic Penrose tiling patch, and 
(c) the same patch with arrows erased. The erasing map $E$ has a local inverse $E^{-1}$.\label{ftiles}}
\end{figure}

The \emph{erasing map} $E: X_{p_\text{5{\rm a}}}\to X_{p_5}$, which erases the arrows, is 
not onto (see \cite{DB}) since no periodic tiling can be correctly decorated with arrows. 
Clearly $E(X_{p_\text{5{\rm a}}})$ is a tiling space.

\begin{definition}
The \emph{unmarked Penrose tiling} tiling space 
is defined by $\sX:=E(X_{p_\text{5{\rm a}}})\subseteq X_{p_5}$, and we call 
the $(\sX,\R^2,T)$ the (unmarked rhombic) \emph{Penrose tiling dynamical system}. 
\end{definition}

\begin{lemma}[deBruijn, \cite{DB}]\label{arrows}
The erasing map $E: X_{p_\text{5{\rm a}}}\to X_{p_5}$
is {\rm 1:1}, and for $\sX=E(X_{p_\text{5{\rm a}}})$,
the inverse $E^{-1}:\sX\to X_{p_\text{5{\rm a}}}$ is an MLD topological  
conjugacy.
\end{lemma} 
\begin{proof} 
It is shown in \cite{DB} that, up to rotation there are only seven vertex configurations 
(see Figure~\ref{vertextypes})
among all the $x\in X_{p_\text{5{\rm a}}}$.  
Types ``S'' and ``S5'' are the only two that are the same in $E(x)$, and since they can never be adjacent,
they can be distinguished in $E(x)$ by the patterns on at least one adjacent vertex. 
\end{proof}

\begin{figure}[h!] 
\includegraphics[width=4.5in]{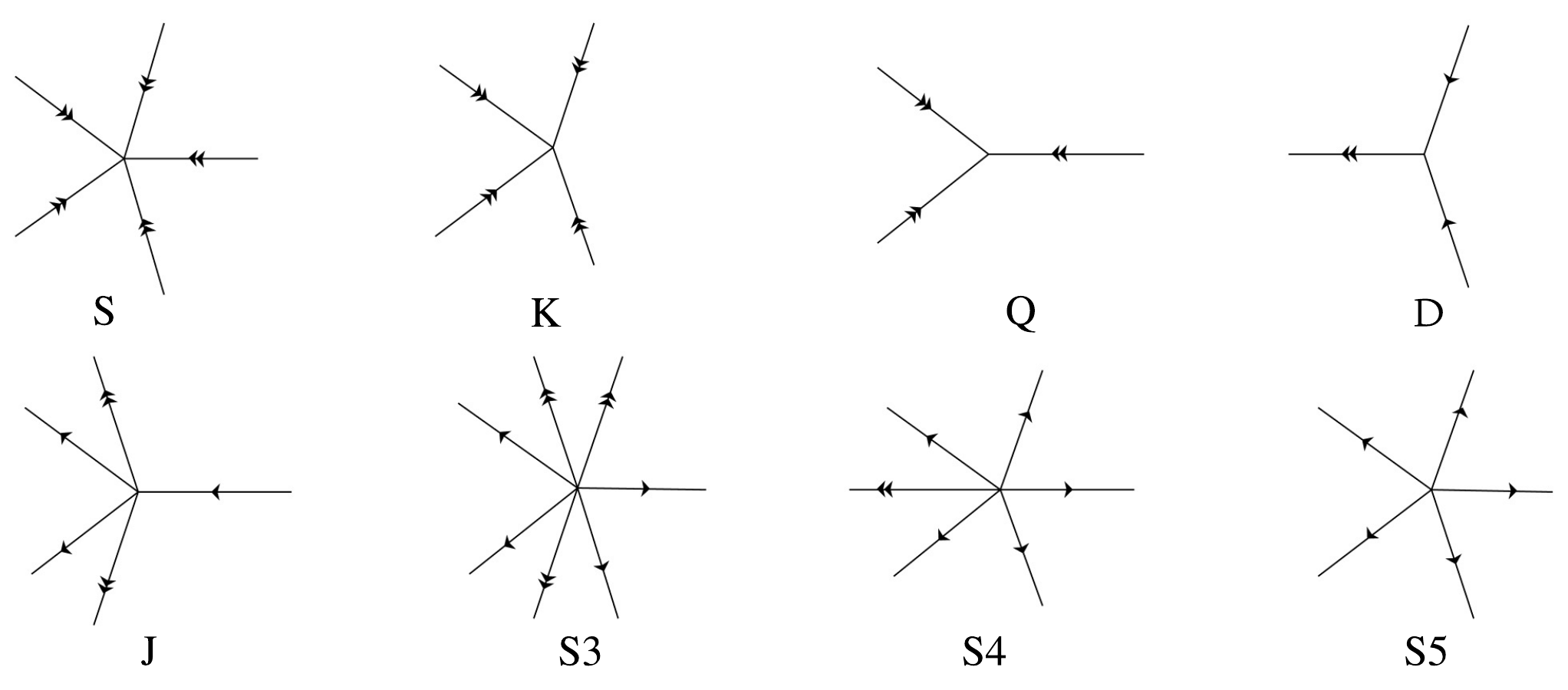}
\caption{Up to rotation there are seven arrowed vertex configurations in all $x\in X_{p_\text{5{\rm a}}}$.
The labels for the types are from 
 \cite{DB}.}\label{vertextypes}
\end{figure}

\subsection{The Kronecker model for Penrose tilings}

Let $\T^5_0=\{\bu\in\T^5: \{u_0+\cdots+u_4\}=0\}$, 
a closed subtorus of $\T^5$, homeomorphic to $\T^4$. 
For $\bv_j=({\rm Re}(e^{2\pi i/5}), {\rm Im}(e^{2 \pi i /5}))$, the $j$th 5th root of unity,  
let $W=[\bv_0^t;\bv_1^t;\dots;\bv_4^t]$ be the $5\times 2$ matrix rows $\bv_j$. 
Let $V$ be the $4\times 2$ matrix obtained by deleting the row of $\bv_3^t$ from $W$.
For $\bt\in\R^2$,  
define $K^\bt \bu=\{\bu+W\bt\}$ on $\T^5$ and $L^{\bt}{\bu}=\{{\bu}+V{\bt}\}$ on $\T^4$.
Clearly $(\T^5,\R^2,K)$ and $(\T^4,\R^2,L)$ are Kronecker systems, with 
$\Sigma_K=\Sigma_L=V^t\Z^5=W^t \Z^4=\Z[e^{2\pi i /5}]\subseteq\R^2$,  the decomplexification of the 
complex ring generated by $e^{2\pi i /5}$. 
Note that $K$ is not minimal since
$\T^5_0$ is $K$ invariant, but $L$ is minimal 
since  
$W^t$ is injective on $\Z^4$. Moreover,  $(\T^4,\R^2,L)$ is topologically conjugate to $(\T^5_0,\R^2,K|_{\T^5_0})$. 
This is 
the unique minimal Kronecker $\R^2$ dynamical system with $\Sigma_K=\Z[e^{2\pi i /5}]$.
The following is the main result of \cite{Robinson1}.

\begin{theorem}[see \cite{Robinson1}]\label{RPenrose}
The Penrose tiling dynamical system $(\sX,\R^2,T)$ is strictly ergodic and
almost automorphic, with
$\Sigma_T=\Z[e^{2\pi i /5}]$ 
 and $\mathfrak{M}_T=\{1,2,10\}$.
\end{theorem}

Previously, Penrose \cite{Penrose} had obtained the basic properties of Penrose tilings,
(many of which also follow from Theorem~\ref{RPenrose}),
although not expressing his results in dynamical terminology.
He showed that every $x\in X_{p_\text{5{\rm a}}}$ is aperiodic, so $T$ is free, and also that 
$T$  is minimal and uniquely ergodic  (see \cite{Robinson1}).
Penrose's proofs of these property are based on the following remarkable scaling property for Penrose tilings.

\begin{proposition}[Penrose, \cite{Penrose}, see also \cite{GS}, \cite{Robinson}]\label{defx}
There is a MLD topological conjugacy {\rm (}see Definition~\ref{local} and the comments that follow it{\rm )}
$C:X_{p_\text{5}}\to X_{(1/\gamma)p_\text{5}}$  {\rm (}or
$C:\sX\to (1/\gamma)\sX$ {\rm )}, where $(1/\gamma)p_\text{5{\rm a}}$  {\rm (}or $(1/\gamma)\sX${\rm )}
is same set of tilings,  scaled down by $1/\gamma$,
 {\rm (}see \cite{Robinson1} for a picture of $C${\rm )}.
\end{proposition}

In addition, the map $C':=\gamma C: X_{p_\text{5{\rm a}}}\to X_{p_\text{5{\rm a}}}$,
or $C':\sX\to\sX$, 
(which scales back up to the original scale)
called a \emph{tiling substitution} or \emph{inflation},
is also a topological conjugacy (see \cite{Robinson}). Note that in \cite{BG} this would be called an \emph{improper}
inflation as opposed to a \emph{stone} inflation, because the map $C$ does not perfectly tile each $D\in p_\text{5}$
with tiles $D'\in (1/\gamma)p_\text{5{\rm a}}$. The scaled map $c'=\gamma c$ 
corresponding to $C':=\gamma C$
(see Remark~\ref{local}) 
can be applied iteratively to tiling patches rather than 
tiles, allowing one to construct ever larger patches thereby proving 
$X_{p_\text{5{\rm a}}}\ne\emptyset$.

 \subsection{Grid tiling dynamical system}

Here we recall some ideas from the proof of Theorem~\ref{RPenrose} from \cite{Robinson1} that we will use in the proof of Theorem~\ref{mainthm}. That proof is based on deBruijn's theory of grid tilings \cite{DB},
(see also \cite{Robinson} or \cite{BG}).

 For each $j=0,1,\dots,4$, $k\in \Z$, and 
 $\bu=(u_0,\dots,u_4)\in \T^5_0$, define the \emph{grid line} 
\begin{equation}\label{gli}
\ell_{j,k}=\ell_{j,k}(u_j):=\{\bv_j\}^\perp+(-u_j+k)\bv_j\subseteq\R^2.
\end{equation} 
We view the set  
$y(\bu):=\{\ell_{j,k}(u_j): j=0,1,\dots,4; k\in\Z\}$ of all grid lines 
as a non FLC tiling of $\R^2$ by convex polygonal \emph{grid tiles}, (see Figure~\ref{nsingul}), 
which we call a \emph{grid tiling}. 
We denote the 
set of all grid tilings by $Y:=\{y(\bu):\bu\in \T^5_0\}$.
\begin{figure}[h!] 
\includegraphics{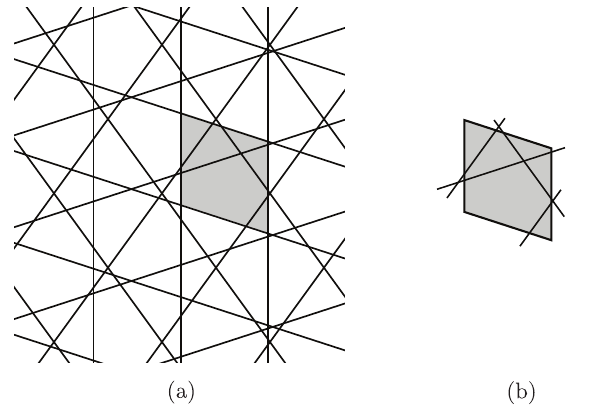}
\captionof{figure}{(a) Patch of nonsingular grid tiling $y(\bu)\in Y_1$. Note that grid tiles 
are often almost vanishingly small. One grid patch $g_\bn(\bu)$ is marked in grey, and (b) 
also shown right. As described in Definition~\ref{sigma}, this 
grid patch has symbol $\sigma(g_\bn(\bu))=0\otimes 1$. 
The dual $w_\bn(x)=g_\bn(\bu)^*$  (see Lemma~\ref{nns}) of this grid patch, where $x=y(\bu)^*$,
is of type  $w^{20}$. It is shown in Figure~\ref{wangpatches}.
\label{nsingul}}
\end{figure} 
The mapping   $\bu\mapsto y(\bu):\T^5_0\to Y$  
is a bijection, continuous in
the topology\footnote{
 The topology on $Y$ 
can also be defined geometrically using a non-FLC version of the tiling metric,
(see \cite{Robinson} or \cite{FS}).
}
that $\T^5_0$ pushes forward to $Y$. Thus
$(Y,\R^2,T)$ is a non FLC tiling dynamical system, where $T$ is the $\R^2$ translation action on $Y$.

\begin{lemma} \label{move}
The map $y:\T^5_0 \to Y$ satisfies  $y(K^\bt \bu)=T^\bt (y(\bu))$.
Thus $(Y,\R^2,T)$ is the unique minimal Kronecker dynamical system 
with spectrum $\Sigma_K=\Z[e^{2\pi i /5}]$. 
\end{lemma}

\begin{proof}
The $\Z^d$ valued function
\begin{equation}\label{coseq}
\bm(\bs,\bu):=\lfloor W \bs+\bu\rfloor-\lfloor \bu\rfloor, 
\end{equation}
where $\bs\in\R^2$, $\bu\in\R^5$,
satisfies $\bm({\bZr},\bu)={\bZr}$, $\bm(\bs,\bu+\bk)=\bm(\bs,\bu)$, $\bk\in \Z^5$,
so is well defined on $\T^5_0$. It also satisfies the 
\emph{cocycle equation} 
$\bm(\bt+\bs,\bu)-\bm(\bs,\bu)=\bm(\bt,K^\bs\bu)$.
Fixing $\bu\in\T^5_0$, we see that $\bm(\bt,\bu)$ is constant (interiors of) 
tiles $E\in y(\bu)$, and takes a different value on each tile.
The left side of the cocycle equation is piecewise constant on the tiles of $y(\bu)-\bs=T^\bs y(\bu)$, 
whereas the right side is piecewise constant on the tiles of $y(K^\bs\bu)$.
\end{proof}

\subsection{Duality and nonsingular tilings}
 
 For an edge to edge tiling $y$ of $\R^2$ by polygonal tiles, let $v(y)$ 
 be its vertices, and $e(y)$ be its edges, with the tiles in $y$ being its faces. 
 Another edge to edge tiling $y^*$  by polygonal tiles is 
 said to be \emph{geometrically dual} to $y$ if there are bijections $y\overset{*}{\leftrightarrow} v(x)$, 
 $v(y)\overset{*}{\leftrightarrow}  x$
 and $e(y)\overset{*}{\leftrightarrow}  e(x)$, such that 
 $\ell^*\in e(x)$, $\ell\in e(y)$ implies  $\ell^*\perp\ell$. 
 
 We will now describe a geometric duality (see \cite{DB}, \cite{Robinson1} for more details)
 between the grid tilings $y=y(\bu)\in Y$
and a certain set of tilings $Y^*\supseteq X_{p5}$. 
 First, for a tile $E\in y(\bu)$, we define the vertex 
 \begin{equation}\label{equivariant}
 E^*:=(2/5) W^t(\bm(\bs,\bu)+\{\bu\})\in v(y^*), 
 \end{equation} 
for arbitrary $\bs\in {\rm int}(E)$.
We connect two vertices $E_1^*,E_2^*\in v(y^*)$ by  
an edge $\ell^* \in e(y^*)$ if their distance is $2/5$, or equivalently, if 
$E_1,E_2\in y$ are adjacent across an edge $\ell\in e(y)$. 
The edges $e(y^*)$ divide $\R^2$ into the polygonal tiles $D$ that make up $y^*$. In particular, 
if $\bb\in v(y)$ is an $n$-fold crossing, then it is the common vertex of  
$2n$ tiles $E_0,E_1,\dots, E_{2n}\in y$, going clockwise. 
The duals $E^*_0,E^*_1,\dots, E^*_{2n}\in v(y^*)$ and the edges in 
$e(y^*)$
connecting 
them enclose
a $2n$-gon $\bb^*:=D\in y^*$. The case $n=2$ is shown in 
Figure~\ref{duals}.

\begin{lemma}[See \cite{Robinson1}]\label{injective}
For $y=y(\bu)\in Y$, 
the map $y(\bu)\mapsto y^*(\bu)$ is injective and  
satisfies
\begin{equation}\label{equi}
(T^\bt y(\bu))^*=y(K^\bt\bu)^*=T^\bt (y(\bu)^*).
\end{equation}
\end{lemma}

\begin{proof}
We first prove  (\ref{equi}).
The first equality follows from Lemma~\ref{move}. For the second, we have that
$v(y(\bu)^*)
=\frac25 W^t(\bm(\R^2,\bu)-\{\bu\})
=\frac25 W^t(\lfloor W\R^2+\bu\rfloor-\bu)$.
Then $K^\bt \bu=W\bt+\bu$ and $\frac25 W^tW=I$ imply 
$v(y(K^\bt \bu)^*)
=\frac25 W^t(\lfloor W\R^2+W\bt+\bu\rfloor- (W\bt+\bu))
=\frac25 W^t(\lfloor W\R^2+\bu\rfloor-\bu)-\bt
=v(y(\bu)^*)-\bt$. 
Since $v(y^*)$ determines $y^*$, we have
$y^*(K^\bt \bu)=y^*(\bu)-\bt=T^\bt y^*(\bu)$.
 
 The duality is clearly injective  up to equivalence: 
if $y_1,y_2\in Y$ are not translates then $y^*_1,y^*_2$ cannot be translates, \cite{DB}.
The injectivity of $y\mapsto y^*$ then follows from (\ref{equi}).
\end{proof}

We call  a grid tiling $y=y(\bu)\in Y$ \emph{nonsingular}
if every vertex $\bb\in v(y)$ is a 2-fold crossing. We 
denote the set of all nonsingular grid tilings 
by $Y_1$. 
Note that $y^{-1}(Y_1)\subseteq\T^5_0$ is dense $G_\delta$, as is  
$Y_1\subseteq Y$.
Now suppose $y=y(\bu)\in Y_1$, and the vertex 
$\bb\in v(y)$ 
is a ($2$-fold) crossing of
 $\ell_{i,k}$ and $\ell_{j,k'}$, $i<j$. Then 
 its  dual $\bb^*$ in $y^*$ is a type $\bv'_i\wedge\bv'_j$ tile in $p_5$ (see  Figure~\ref{duals}). 
 Thus $y\in Y_1$ implies
 $y^*\in X_{p_5}$.
 But in fact, much more is true.
 
 \begin{theorem}[deBruijn, \cite{DB}]
If $y=y(\bu)\in Y_1$ then  $x(\bu):=y(\bu)^*\in \sX$. In other words, the dual  
$x(\bu)$ of any nonsingular grid tiling $y(\bu)$ is an unmarked  Penrose tiling.
\end{theorem}

\begin{figure}[h!]
\includegraphics[width=3.9in]{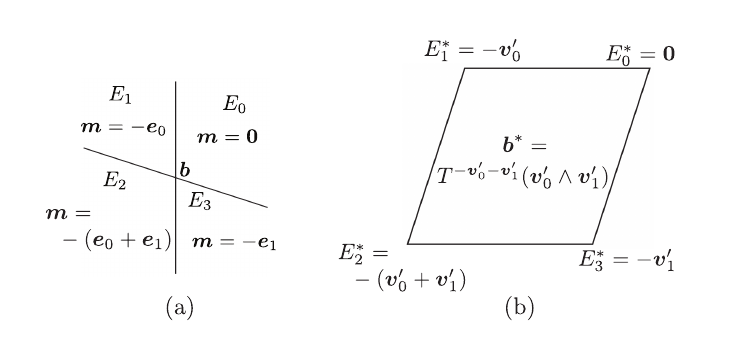}
\caption{
A vertex $\bb\approx\bZr\in v(y)$ (left) that is a 2-fold crossing of a 0-grid line and a 1-grid line. The values 
of $\bm(\bs,\bu)$ on the four 
tiles  $E_0,E_1,E_2$ and $E_3$ adjacent to $\bb$ have values 
$\bm={\bZr}$, $\bm=-\be_1$,   
$\bm=-(\be_0+\be_1)$ and $\bm=-\be_1$. 
Assuming $\{\bu\}=\bZr$, their 
duals (right) are $E_0^*={\bZr}$, 
$E_1^*=-\bv'_0$, 
$E_2^*=-(\bv'_0+\bv'_1)$, 
$E_3^*=-\bv'_1$,
 which are the vertices of a tile $\bb^*$ (right), of type $\bv_0'\wedge\bv_1'$.
\label{duals}}
\end{figure}

Let $\sX_1:=Y^*_1$, called  the \emph{nonsingular} Penrose tilings, and let 
$\varphi: \sX_1\to Y_1$ be the inverse of the duality map. 

\begin{lemma}[See \cite{Robinson1}] \label{closure}
The map
$\varphi:\sX_1\to Y_1$ is uniformly continuous and has a continuous extension
to $\varphi:\sX\to Y$ where  
$\sX:=\overline{\sX_1}$. 
\end{lemma}

This is essentially
deBruijn's observation \cite{DB} that all Penrose tilings are limits of nonsingular Penrose tilings. The 
extension $\varphi:\sX\to Y$ is the factor map in Theorem~\ref{RPenrose}.
The fact that $\varphi$ has a single preimage on the $G_\delta$ set $Y_1$ shows that $\varphi$ is almost 1:1.

\subsection{The singular cases}\label{singularcases}
 
\begin{figure} [h!]
\includegraphics[width=4.5in]{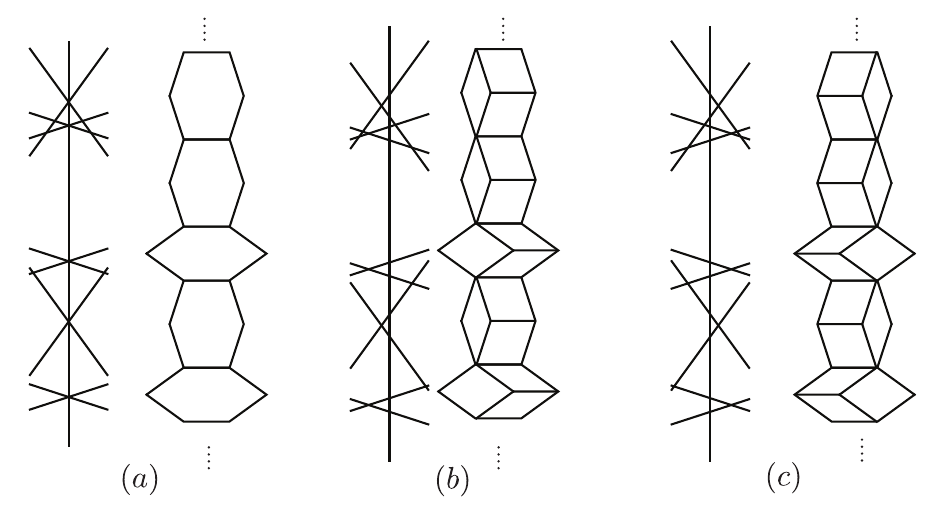}
\caption{(a) The $3$-fold crossings along the \emph{spine} $\ell_0$ for $y(\bu)\in Y_2$
and the unfilled work in its dual $y(\bu)^*$. (b) and (c) show the $+$ and $-$ resolutions 
of the worm.
\label{worm}
} 
\end{figure}
 
If $y(\bu)\notin Y_1$ it means that some of crossings in 
$y(\bu)$ are  not $2$-fold. We call $y(\bu)$ a 
\emph{singular} grid tiling. Let $n$ be the 
highest crossing order in $y(\bu)$. 
It is easy to see that $\bu\in\T^5_0$ implies that $n=4$ never occurs (see  \cite{DB}).
Thus there are two cases: {\bf Case A}, $n=3$, which we will see corresponds to $Y_2$ in Theorem~\ref{RPenrose}, and 
{\bf Case B}, $n=5$, 
which corresponds to $Y_{10}$. 

{\bf Case A:} 
Let  $\bb$ be one of the 3-fold crossings in $y(\bu)$. 
As shown in \cite{DB} (see also \cite{Robinson1}), there is a grid line $\ell_{j,k}$ through 
$\bb$, we call it a \emph{spine}, 
such that all the crossings along $\ell_{j,k}$ are $3$-fold, and all the other crossings 
in $y(\bu)$ are $2$-fold. The $3$-fold crossings along  $\ell_{j,k}$ are of two kinds:  \emph{wide} 
is a crossing of $\ell_j$ with (some) $\ell_{j+2,k'}$ and $\ell_{j-2,k''}$ (where the arithmetic for $j$ is $\!\!\!\mod\!5$), and 
 \emph{narrow} 
is a crossing of $\ell_j$ with $\ell_{j+1,k'}$ and $\ell_{j-1,k''}$, (see Figure~\ref{worm}(a)).

The duals of
$3$-fold crossings are hexagons, respectively \emph{wide} or \emph{narrow}. The hexagons 
line up in $y(\bu)^*$ in configuration called an \emph{unfilled worm}, 
Figure~\ref{worm}(b), and we call the line $l$ along the center of the worm its \emph{axis}.
Since hexagons are not in $p_5$, the dual $y(\bu)^*$ is not in $X_{p_5}$. 
We can obtain a tiling in $X_{p_5}$ by ``filling'' each hexagon in $y(\bu)^*$
by tiling it with three tiles from $p_5$. This can be done in one of  two \emph{directions}: 
direction $+$,  meaning a single tile sits on the side facing $\bv_j$, and  
direction $-$, meaning two tiles sit on the side facing $\bv_j$, (see Figure~\ref{worm}(b) and (c)).

To obtain a Penrose tiling, all the hexagons must be filled in the same direction \cite{DB}.
We denote the two resulting Penrose tilings as $x^+(\bu), x^-(\bu)\in \sX$, Figure~\ref{worm}(b) and (c). 
The column of filled hexagons in $x^+(\bu)$ or $x^-(\bu)$ is called a 
\emph{filled worm}. Going back and forth between 
$x^+(\bu)$ and $x^-(\bu)$ is called \emph{flipping} the worm. Except for the tiles 
in the filled worm, all the tiles in 
$x^+(\bu)$ and $x^-(\bu)$ are the same. The following is the first part of the proof of Theorem~\ref{mainthm}.

 \begin{proposition}\label{firstpart}
If $\cV$ is a direction perpendicular to a $5$th root of unity,  
then the Penrose tiling dynamical system 
 $(X,\R^2,T)$ is not expansive in the direction $\cV$.
 \end{proposition}
 
\begin{proof} 
Suppose $\cV=\{\bv_j\}^\perp$.
Let $x^+\in \sX_2$ and  fix $r>0$.
By applying a rotation and a translation we can assume that 
the single worm in $x^+$ is in  a direction $\bv_j^\perp$ that
lies outside $\cV^r$. Let $x^-$ be 
$x^+$ with the 
worm flipped. Then $x^+[\cV^r]=x^-[\cV^r]$ but $x^+\ne x^-$. 
It follows from Definition~\ref{exp0} and Proposition~\ref{exp} that 
$(\sX,\R^2,T)$ is not expansive in the direction $\cV$.
\end{proof}

\begin{figure}[h]
\includegraphics{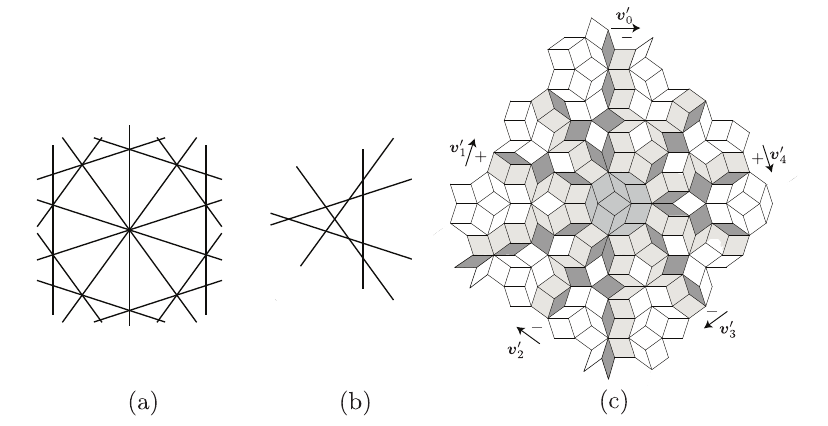}
\caption{(a) A $5$-fold crossing in $y(\bu)\in Y_{10}$ with some surrounding $3$-fold crossings, (b) a small perturbation 
$y(\bu')\in Y_1$ of the 
$5$-fold crossing, and 
(c) the decagon tiling in rotation $0$ together with worm filling directions: $-,+,-,-,+$. 
\label{cartwheel}}
\end{figure}
 
{\bf Case B:} Here,  $y=y(\bu)$ has a single $5$-fold crossing at
$\bb\in \R^2$. By translation, we can assume $\bb=\bZr$ so the \emph{spines} are $\ell_{j,0}$, $j=0,\dots,4$.
All the crossings along these spines, except at $\bZr$, are wide and narrow
$3$-fold, and all the other crossings in 
$y(\bu)$ are $2$-fold 
(see \cite{DB}).
In $y(\bu)^*$  
the dual of the $5$-fold crossing is a regular decagon, and there are unfilled half-worms of hexagons radiating 
from its faces. This configuration is called an 
\emph{unfilled cartwheel}, and we call a direction $\bv_j^\perp$ line $l_j$ a \emph{spoke}
if it goes through the middle of a pair of half-worms.  

Clearly $y(\bu)^*$ is invariant under rotations around $\bZr$ by multiples of $2\pi/10$.
We fill the unfilled cartwheel with $p_5$ tiles to obtain 
a Penrose tiling $x(\bu)^0\in\sX$ as follows. We fill the half-worm pairs, in the order $\bv_0^\perp, \dots, \bv_4^\perp$, 
with a pattern of directions $-,+,-,-,+$. Each counter-clockwise rotation of $x(\bu)^0$ 
by $2\pi/10$ acts on the sign pattern by 
flipping all the directions and cycling the pattern by moving 
the last direction to the front. So, for example, a single rotation of $x(\bu)^0$, which we 
denote $x(\bu)^1$, has 
the sign pattern $-,+,-,+,+$. There are $10$ such patterns which yield the 10 distinct rotations of the 
decagon filling, and yield $10$ distinct Penrose tilings $x(\bu)^0,\dots, x_9(\bu)\in \sX$, 
obtained from filling the unfilled cartwheel in 
$y(\bu)\in Y_{10}$. Off this unfilled cartwheel these are all the same. 
These $10$ Penrose tilings are the 
$10$ distinct $\varphi$ preimages of $y(\bu)\in Y_{10}$, (see \cite{DB}).

\section{Proof of Theorem~\ref{mainthm}}

\subsection{Sturmian sequences}\label{sturmianseq}

Recall that  $\alpha:=\gamma-1\approx 0.6180$. The slope $\alpha$ 
\emph{irrational rotation} is the Kronecker  system $(\T,\Z,R_\alpha)$ defined by
 $R_\alpha u:=u+\alpha\!\!\mod 1$. 
   It satisfies 
  $\Sigma_{R_\alpha}=\Z[\alpha]/\Z\subseteq\T\sim\widehat{\Z}$, where $\Z[\alpha]$ is
   the ring generated by $\alpha$, and $\Z[\alpha]/\Z$ is its projection to the circle.
   A \emph{Sturmian sequence} $z=(z_n)_{n\in\mathbb Z}\in \{0,1\}^{\mathbb Z}$ is a coding of $u\in \T$
by $R_\alpha$. It is defined by
$z_n=z(u)_n:=\chi_I(R^n_\alpha u)$, where 
 $I$ is either $I^+:=[1-\alpha,1)$ or $I^-:=(1-\alpha,1]$. We let $Z=\{(z(u)_n)_{n\in\mathbb Z}:u\in\T\}$, 
 with $\phi:Z\to\T$ being the factor map satisfying $\phi(z(u))=u$.
 The following lemma summarizes some well known facts.

 \begin{lemma}\label{sturml}
 The  $\Z$ Sturmian subshift $(Z,\Z,S)$ {\rm (}not  SFT{\rm )}
 is  strictly ergodic almost automorphic with
 $\Sigma_S =\Z[\alpha]/\Z$, 
 $\mathfrak{M}_T=\{1,2\}$, and  $\T_1:=\T\backslash\{R^n_\alpha 0:n\in\Z\}$.
 \end{lemma}
 
Most of this can be found in \cite{Fogg}, with  explicit formula for $\phi(z)$  in \cite{RobiSturmianentropy}.
We call the tensor square $(Z\otimes Z,\Z^2,S\otimes S)$ of $(Z,\Z,S)$
the $\Z^2$ \emph{Sturmian subshift}.

\begin{lemma}\label{stfact}
The $\Z^2$ Sturmian subshift $(Z\otimes Z,\Z^2,S\otimes S)$
 is  strictly ergodic almost automorphic with 
  $\Sigma_{S\otimes S} =(Z[\alpha]/\Z)\times (Z[\alpha]/\Z)$, $($i.e., an almost $1:1$ etension of $R_\alpha\otimes R_\alpha$$)$, 
 $\mathfrak{M}_T=\{1,2,4\}$, and  
 $\T^2_1:=\T^2\backslash\{(R_\alpha\otimes R_\alpha)^\bn \bZr:\bn\in\Z^2\}$.
 The factor map is $\phi\otimes\phi$.
\end{lemma}

\subsection{Grid patches}\label{gridpatches}
Let 
\begin{align}
{\bbf}_0&:=(0,\sec(2\pi/20))=(0,1+\sqrt{5})\approx (0,3.236),{\rm\  and} \nonumber \\
{\bbf}_1&:=(1,-\tan(2\pi/20))=(1,\sqrt{5+2\sqrt{5}})\approx (-1,3.078), \label{fs}
\end{align}
 and let $A$  be the matrix with 
$\bbf_0$ and $\bbf_1$ as columns. 
We will define a rhombic tile $R$ by 
\begin{equation}\label{rhombs}
R:=\bbf_1\wedge\bbf_0=\{t_1{\bbf}_1+t_0{\bbf}_0:t_0,t_1\in[0,1]\}\subseteq{\R}^2=A[0,1]^2.
\end{equation}
For any $\bu=(u_0,\dots,u_5)\in\T^5_0$, 
the $0$ and $1$-grid lines in $y(\bu)$ 
partition ${\mathbb R}^2$ into a 
tessellation $\{R_\bn\}_{\bn\in\Z^2}$ by congruent rhombs:
\begin{equation}\label{tess}
R_\bn:=T^{-((-u_0+n_0 ) \bbf_0+(-u_1+n_1 ){\bbf}_1)}R=R+A(\bn-\bu'), 
\end{equation}
where in the second equality, we assume without loss of generality that $\bu':=(u_0,u_1)\in[0,1)^2$.

Note that ${\bZr}\in R_{\bZr}$.
We define the \emph{grid patch}
 $g_{\bn}({\bu})$ 
to be the patch in $y({\bu})$ 
supported on the rhombus $R_{\bn}$
(see Figure~\ref{nsingul}(b)).

\begin{definition}\label{sigma}
Let 
$y(\bu)\in Y_1$ be a nonsingular grid tiling (all crossings $2$-fold). 
For ${\bn}=(n_0,n_1)\in\mathbb Z^2$, define  the \emph{symbol of the grid patch} $g_{\bn}({\bu})$, denoted 
by $\sigma(g_{\bn}({\bu}))=z_{n_0}\otimes z'_{n_1}\in\{0\otimes 0,0\otimes 1,1\otimes 0,1\otimes 1\}$
as follows: 
If
a single $4$-grid line passes through $R_{\bn}$,  put $z_{n_0}=0$, but 
if two $4$-grid lines pass through $R_{\bn}$,  put $z_{n_0}=1$.
Similarly, if a single $2$-grid line passes through $R_{\bn}$,  put $z'_{n_1}=0$, but 
if two $2$-grid lines pass through $R_{\bn}$,  put $z'_{n_1}=1$.
\end{definition}

\begin{lemma}\label{read}
Suppose that ${\bu}\in\mathbb T^5_{0,1}$ satisfies  $u_0=u_1=0$. 
Then $z_{0}\otimes z'_{0}=\sigma(g_{\bZr}({\bu}))$
is given by $z_{0}=\chi_{[1-\alpha,1)}(u_4)$ and $z'_{0}=\chi_{[1-\alpha,1)}(u_2)$.
Moreover, $\{\sigma(g_{\bn}({\bu}))\}_{\bn\in\Z^2}=z(u_4)\otimes z(u_2)$.
\end{lemma}

\begin{proof}
First note that $g_{\bZr}({\bu})$ is supported on $R_{\bZr}$.
A calculation (\ref{gli}) 
shows that $\ell_{2,0}(1-\alpha)$  passes through the upper right vertex of $R_{\bZr}$. 
If $z'_0=0$, Definition~\ref{sigma} says that $\ell_{2,0}(u_2)$ is the only $2$-grid line through 
$R_{\bZr}$. Since it enters $R_{\bZr}$ through  the bottom edge (and exits through the top),
it lies to the left of $\ell_{2,0}(1-\alpha)$. Thus $u_2\in[0,1-\alpha)$. 
Likewise, if $z'_{0}=1$, there is a second 2-grid line 
$\ell_{2,-1}(u_2)$ that enters 
$R_{\bZr}$ through the left edge and exits through the top.  
This means $\ell_{2,0}(u_2)$ enters  $R_{\bZr}$ to the right of  $\ell_{2,0}(1-\alpha)$
(and exits through the right edge).
In both cases, we have $z'_0=\chi_{[1-\alpha,1)}(u_2)$.
 The same argument shows $z_0=\chi_{[1-\alpha,1)}(u_4)$. 
 See Figure~\ref{cutlines}.
 
 \begin{figure}[h!]
 \begin{center}
 \includegraphics{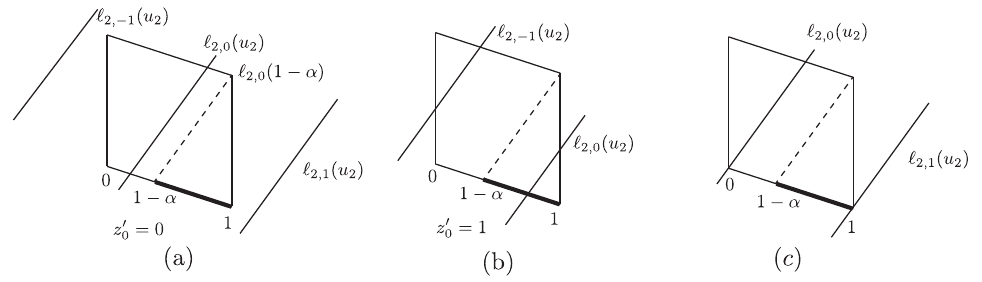}
 \vskip -.15 in 
 \caption{ (a) In the case $z'_0=0$, a single 2-grid  line,  $\ell_{2,0}(u_2)$, 
 crosses $R_{\bZr}$, so that $u_2\in [0,1-\alpha)$ since $\ell_{2,0}(u_2)$ is left of $\ell_{2,0}(1-\alpha)$. (b) 
  In the case $z'_0=1$, two 2-grid lines cross $R_{\bZr}$ and $u_2\in [1-\alpha,1)$. (c) The singular case
  ($u_2=0$). 
  \label{cutlines}}
 \end{center}
 \end{figure}
 
 For the second part, $u_0=u_1=0$ implies 
  $R_\bn=T^{n_0{\bbf}_1+n_1{\bbf}_0}R$, so
$z_{n_0}\otimes z'_{n_1}=\sigma(g_{\bZr}(K^{-\bs_\bn}{\bu}))
=\sigma(g_{\bZr}(\{-W\bs_\bn+\bu\})$. 
By a computation, $W{\bbf}_0=(0,1,\alpha,-\alpha,-1)$ and $W{\bbf}_1=(1,0,-1,-\alpha,\alpha)$, and 
$-W\bs_\bn+\bu=(n_0,n_1,-n_0+n_1\alpha+u_2,-(n_0+n_1)\alpha+u_3,n_0\alpha-n_1+u_4)$, 
since $\bu=(0,0,u_2,u_3,u_4)$.
By the definition of a Sturmian sequence, $z'_{n_1}=\chi_{[1-\alpha,1)}(\{n_1\alpha+u_2\}) 
=\chi_{[1-\alpha,1)}(R_\alpha^{n_1}u_2)={z}(u_2)_{n_1}$. 
The same argument shows $z_{n_0}={z}(u_4)_{n_0}$.
\end{proof}
 
 \subsection{Duals of grid patches}\label{dgp}

\begin{lemma}\label{movement}
Let $x\in \sX_1$ be a nonsingular Penrose tiling with $x=y(\bu)^*$ 
for $y=y(\bu)\in Y_1$. Suppose that $\bb_\bn\in v(y)$ satisfies 
$\bb_\bn=\ell_{0,n_0}\cap\ell_{1,n_1}$, $\bn=(n_0,n_1)$, so 
$\bb_\bn$ is the lower left vertex of the rhombus $R_\bn$ in {\rm (\ref{tess})}. 
Let $E_\bn\in y$ be the tile in $g_\bn(\bu)$ with vertex $\bb_\bn$. 
The dual $D_\bn=\bb_\bn^*\in x$ of $\bb_\bn$ is a type $\bv'_0\wedge\bv'_1$ tile from $p_5$ in $x$, and the 
dual $\bd_\bn=E_\bn^*$ of the tile $E_\bn\in y$ is the upper right vertex of $D_\bn$.
Then $||\bb_\bn-\bd_\bn||<(1/5)(1+\sqrt{5})\approx 0.647$.
\end{lemma}

\begin{proof}
Use Lemma~\ref{injective} to translate $\bb_\bn$ to the origin, and making it $\bb_{\bf 0}$,
 by replacing $\bu$ with $K^{-\bb_\bn}\bu$. Then $\{\bu\}=(0,0,u_2,\{-(u_2+u_4)\},u_4)$ with $u_2,u_4\in[0,1)$.  
We show  $||\bd_{\bf 0}||<(1/2)(1+\sqrt{5})$. By (\ref{equivariant}),  
$\bd_{\bf 0}=E^*_{\bf 0}:=(2/5) W^t(\bm(\bs,\bu)+\{\bu\})$ for any 
$\bs\in{\rm int}(E_{\bf 0})$. Making $\bs>0$ small, we see 
from (\ref{coseq}) that $\bm(\bs,\bu)={\bZr}$ for $\bs\in{\rm int}(E_{\bf 0})$,
since $\bm$ is constant on tiles. 
Thus $||\bd_{\bf 0}||=(2/5) ||W^t\{\bu\}||\le (2/5)||\bv_2+\bv_3+\bv_4||=(2/5)(1+\sqrt{5})/2=(1/5)(1+\sqrt{5})$.
\end{proof}

Let $x\in\sX$ (or more generally, $x\in X_{p_5}$).
A  $\bv_j^\perp$-\emph{trail} in $x$ is
an infinite sequence of tiles  $\{\dots, D_{-1},D_0,D_1,\dots\}$  in $x$
such that each $D_{i+1}$    is 
 adjacent across to $D_i$  across a type  $\bv'_j$ edge in the direction $R_{2\pi/4}\bv_j$. 
 For $x=y(\bu)^*\in \sX_1$  every  $\bv_j^\perp$-trail
 is the dual of the sequence of $2$-fold crossings along a grid line 
$\ell_{j,k}(u_k)$ in $y(\bu)$ for some $k\in\Z$.
 A  $\bv_j^\perp$-\emph{trail segment} is a contiguous finite subsequence
 $\{D_{k},D_{k+1},\dots,D_{k+n-1}\}$ of a $\bv_j^\perp$-trail.

\begin{lemma}\label{pseudolines}
Fix a Penrose tiling $x\in \sX$ and let $x'$ be the set of all 
type $\bv_0'\wedge \bv_1'$ tiles in $D\in x$.
 Then there is a bijective map $\bn\mapsto D_\bn:\Z^2\to x'$ 
such that for each $\bn\in\Z^2$, 
a $\bv_0^\perp$-trail segment connects $D_\bn$ to $D_{\bn+\be_1}$
and a $\bv_1^\perp$-trail segment connects $D_\bn$ to $D_{\bn+\be_0}$.
  The map $\bn\mapsto D_\bn$ is unique for $x\in \sX_1$ if we assume $D_{\bf 0}=\bb_{\bf 0}^*$, 
  or for $x\in \sX\backslash \sX_1$  if we assume  $D_{\bf 0}=(\bb'_{\bf 0})^*$ 
 with $x'\in\sX_1$ and $\rho(x,x')$ sufficiently small.  
 \end{lemma}

 \begin{proof} 
If $x\in \sX_1$ is nonsingular, the map $\bn\mapsto D_\bn$ defined in Lemma~\ref{movement} satisfies the lemma.
If $x\in \sX\backslash \sX_1$ is singular, there is a sequence $x_k\in \sX_1$ so that $x_k\to x$  in the metric 
$\rho$, (i.e., $\sX=\overline{\sX_1}$, Lemma~\ref{closure}). Let $x_k=y(\bu_k)^*$ and let  
$\bn\mapsto D^k_\bn$ be the map defined in Lemma~\ref{movement}. It 
satisfies $D^k_\bn=\bb^k_\bn$ for $\bb^k_\bn\in v(y(\bu_k))$. 
Let $r_n>0$ be an increasing sequence, and for each $n$, choose $K_n$ large enough that $k\ge K_n$ implies 
$x_k[B_{r_n}({\bf 0})]=x[B_{r_n}({\bf 0})]$. 
Define $D_\bn:=D^k_\bn$ 
where $r_n>||\bn||$ and $k\ge K_n$.
\end{proof}

\begin{definition}
For  a Penrose tiling $x\in \sX$ and $\bn\in \Z^2$, we define a patch 
$w_\bn(x)$ in $x$, called a \emph{Wang patch}, as follows: 
Starting at $D_\bn$, go forward along the  $\bv^\perp_0$-trail segment
to $D_{\bn+\be_0}$ (Lemma~\ref{pseudolines}), then
go forward along the  $\bv^\perp_1$-trail segment 
to $D_{\bn+\be_0+\be_1}$, then go backward along 
the  $\bv^\perp_0$-trail segment
to $D_{\bn+\be_1}$, and finally go backward along 
the  $\bv^\perp_1$-trail segment
 back  to $D_\bn$. The patch  
$w_\bn(x)$ is defined to be this loop of closed 
trail segments,
together with all the tiles inside it.
\end{definition}

The Wang patches $\{w_\bn(x)\}_{\bn\in \Z^2}$ in $x\in\sX$
cover $\R^2$ in the sense that $x=\bigcup_{\bn\in\Z^2}w_\bn(x)$.
We call the four type $\bv_0'\wedge \bv_1'$ tiles in $w_\bn(x)$
its corners. The following is a consequence of Lemma~\ref{pseudolines}.

 \begin{lemma}\label{nns}
 For $x\in \sX_1$, let $x=y(\bu)^*$ where $y(\bu)\in Y_1$.
 Then each Wang patch  in $x$ is the dual of the corresponding grid patch: $w_\bn(x)=g_\bn(\bu)^*$. 
 \end{lemma}
 
\begin{figure}[h!]
\includegraphics[width=5.1in]{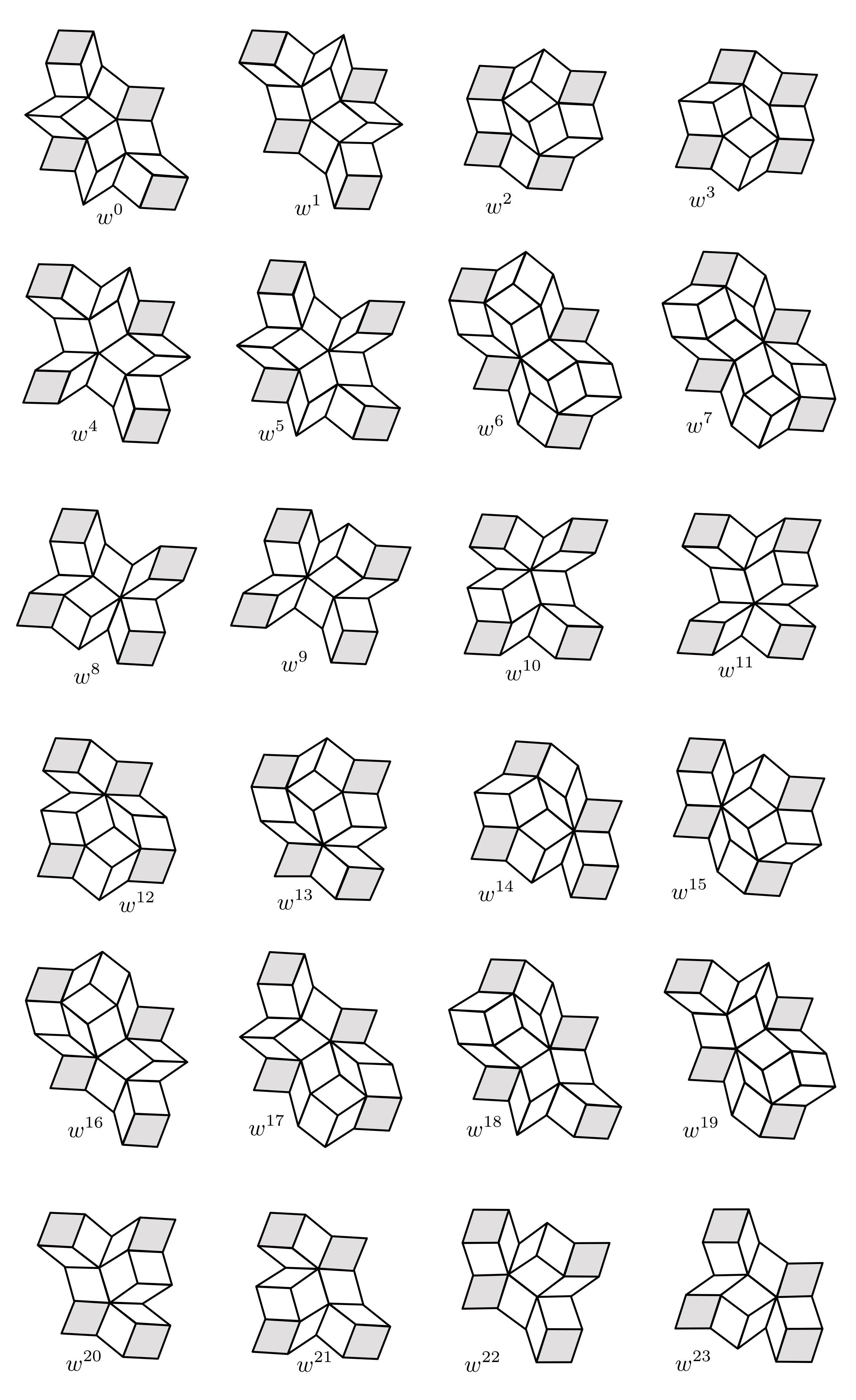}
\caption{The $24$ Wang patches.  
\label{wangpatches}} 
\end{figure} 

\begin{proposition}[Raphael Robinson, see Exercise 11.1.2 in \cite{DB}]\label{rr24}
Up to translation, there are $24$ Wang (proto-) patches among all the Penrose tilings $x\in\sX$. 
\end{proposition}

We denote
these Wang (proto-) patches by $\{w^a:a\in\cA\}$ where $\cA=\{0,\dots,23\}$.
These $24$ Wang patches are shown in Figure~\ref{wangpatches}. Pictures of these patches appeared in the 
first author's dissertation \cite{Jang}, and also in \cite{mann}.

\begin{figure}[h!]
\includegraphics[width=8cm]{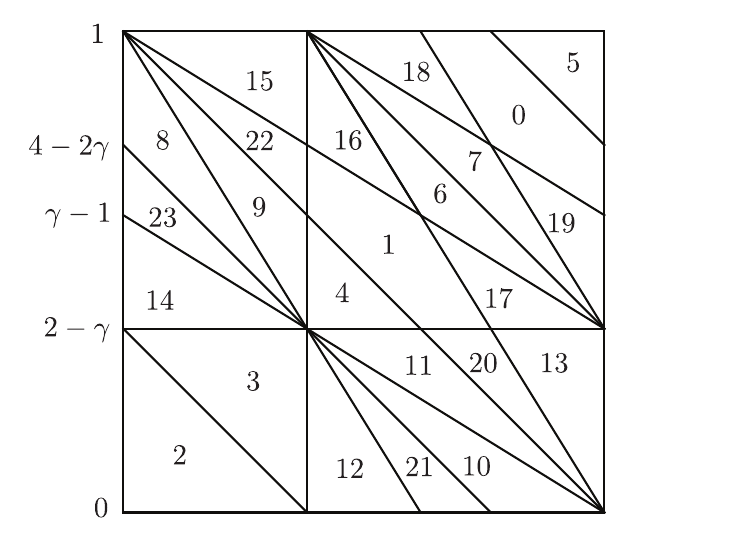}
\caption{
The bifurcation diagram,  showing the parameter values $(u_2,u_4)\in\T^2\sim [0,1]^2$ 
corresponding the 24 Wang patches. Each cell is labeled 
with $a\in\cA$ for the corresponding Wang patch $w^a$ (Figure~\ref{wangpatches}). 
The vertices along each edges of the 
square are at $0$,
$2-\gamma\approx 0.381967$, $\gamma-1\approx 0.618034$, 
$4-2\gamma\approx 0.763932$, and $1$.
The coarser division of $[0,1]^2$ into four cells 
by the lines $u_2=2-\gamma$ and $u_4=2-\gamma$ corresponds to the values 
$0\otimes 0, 0\otimes 1, 1\otimes 0, 1\otimes 1$  of the 
symbol $\sigma(y(\bu))$ where $\bu=(0,0,u_2,u_3,u_4)$ with $u_2+u_3+u_4=0$, 
corresponding to the $\Z^2$ Sturmian shift $S\otimes S$. 
\label{markov}} 
\end{figure} 

\begin{proof}[Proof of Proposition~\ref{rr24}]
By Lemma~\ref{nns}, all the Wang patches $w^a$, up to translation, can be obtained as duals 
$g_\bn(\bu)^*$ of all the grid patches $g_\bn(\bu)$ that occur in all $y(\bu)\in Y_1$.
This means 
all crossings in $y(\bu)$ are 2-fold, and in particular, this holds for any grid patch in 
$g_\bn(\bu)$ in $y(\bu)$.  We will call a 
grid patch that has only 2-fold crossings a \emph{nonsingular grid patch}.
Now, if $g_\bn(\bu)$ is a nonsingular grid patch in $y(\bu)\in Y$, (i.e., $y(\bu)$ is not 
necessarily nonsingular), then since $Y_1$ is dense $G_\delta$,
there exists a small perturbation of $\bu$, which does not change 
$g_\bn(\bu)$ enough to change its dual
$g_\bn(\bu)^*$, so that with this new $\bu$, we have $y(\bu)\in Y_1$.

Given any grid patch $g_\bn(\bu)$ in $y(\bu)\in Y$ we can 
translate it to the origin, using Lemma~\ref{injective} to adjust $\bu$,
so that  for this new $\bu$, the patch $g_\bn(\bu)$ is translated to 
the origin and becomes $g_\bZr(\bu)$, with its support satisfying
$R_\bZr=R$, where $R$ is as in (\ref{rhombs}).
The new value of $\bu$ thus has the form $\bu=(0,0,u_2,u_3,u_4)$, where 
we  assume without loss of generality 
that $(u_2,u_3)\in [0,1]^2$ with $u_2+u_3+u_4=0$. Thus we  think of  $g_\bZr(\bu)$ as 
depending on $(u_2,u_4)\in [0,1]^2$,  but  we also observe that
$g_\bZr(\bu)$ is the same for $(u_2,0)$ and $(u_2,1)$, and 
for $(0, u_4)$ and $(1, u_4)$, so we can also regard $(u_2,u_4)\in\T^2$.

Figure~\ref{markov} shows a partition of  $[0,1]^2$ 
into $24$ cells, $46$ edges, and $23$ vertices. 
We call this the \emph{bifurcation diagram} for Wang patches, noting that  
the same diagram, 
called a \emph{Markov partition} for Wang tiles, 
was obtained in \cite{mann}. 
The cells in the bifurcation diagram are defined by the condition that a point
 $(u_2,u_4)$ is in the interior of a cell if and only if the corresponding grid patch  
$g_{\bZr}(\bu)$ is nonsingular. 
It follows that
the dual $g_{\bZr}(\bu)^*$ is a constant Wang patch $w^a$ 
for all  $(u_2,u_4)$ in the interior of each cell.
By the discussion in the previous two paragraphs, 
the duals $g_{\bZr}(\bu)^*$ range over all possible Wang patches $w^a$ up to translation.
We can see each of these Wang patches by   
drawing the dual for the grid patch $g_\bZr(\bu_0)$  corresponding to a single 
point $\bu_0$ in each cell. 
We find that each cell results in a distinct Wang patch.
Thus there are  
24 distinct Wang patches, which 
are the Wang patches $w^a$, $a\in\cA$, shown in 
Figure~\ref{wangpatches}. The numbering $a\in \cA$ corresponds 
to the labels of the cells in Figure~\ref{markov}.  
In particular, these are the 24 Wang patches  described in \cite{GS}.

It remains to show that the bifurcation diagram is correct. 
Figure~\ref{markov} 
shows  five vertices along each side of $[0,1]^2$ located
at $\eta_0=0$, $\eta_1:=2-\gamma\approx 0.381967$, $\eta_2:=\gamma-1\approx 0.618034$, 
$\eta_3:=4-2\gamma=3-\sqrt{5}\approx 0.763932$, and $\eta_4=1$.
We will now locate all the possible multiple crossings in $g_\bZr(\bu)$. 
The remainder of the proof is organized in five steps.

\smallskip

\noindent {\it Step} 1. It is easy to see  that there are $5$-fold crossings $g_\bZr(\bu)$ exactly when $(u_2,u_4)$ 
in $[0,1]^2$
is one of the nine vertices 
corresponding to the partition of $[0,1]^2$ 
into four cells by the vertical line $u_2=\eta_1$ and the horizontal line $u_4=\eta_1$,
(in $\T^2$ these identify to four vertices). 
In subsequent discussions 
we implicitly ignore these nine points, and look only for $3$-fold crossings
in $g_\bZr(\bu)$.

Altogether, there are ten possible types of $3$-fold crossings, which 
we denote $012$, (meaning a crossing of a $0$-, $1$- and a 
$2$-grid line), $013$, $014$, $023$, $024$, $034$
$123$, $124$, $134$, and $234$. 

\smallskip

\noindent {\it Step} 2. It is easy to see that if $(u_2,u_4)$ lies along the 
left or right side 
of $[0,1]^2$ (on the line $u_2=0$ or $u_2=1$) then there are $2$-grid lines through 
the lower left and lower right corners of $R_\bZr$ in $g_\bZr(\bu)$, as shown in Figure~\ref{cutlines}(c). 
Similarly, for points along the 
vertical line $u_2=\eta_1$, there are $2$-grid lines through 
the upper left and upper right corners of $R_\bZr$. These account for all 
possible $3$-fold crossings type of $012$ in $g_\bZr(\bu)$. In a similar way, for 
$(u_2,u_4)$ 
along any of the three horizontal lines $u_2=0$, $u_2=\eta_1$ or $u_2=1$,  
there are   
$3$-fold crossings of type $014$ at corners of $R_\bZr$ in $g_\bZr(\bu)$,  
and there can be no other crossings of this type.

\smallskip

\noindent{\it Step} 3. For all $(u_2,u_4)$ along the five lines of slope $-1$ in the bifurcation diagram, 
(i.e., the lines $u_2+u_4=\eta_1$, $u_2+u_4=\eta_2$, $u_2+u_4=1$, $u_2+u_4=2\eta_2$ and $u_2+u_4=2\eta_3$)
there  are $3$-fold crossings type $013$ at corners of  $R_\bZr$. For two of these, 
$u_2+u_4=\eta_1$ and $u_2+u_4=2\eta_2$, the $3$-grid lines cross the upper left and lower right corners.
For the other three, the $3$-grid lines cross the lower left and upper right corners of $R_\bZr$.
No other crossings of  type $013$ are possible. 

\smallskip

\noindent{\it Step} 4. 
We show that for $(u_2,u_4)\in[0,1]^2$, all possible type $234$ crossings occur  
along (two of the) slope $-1$ lines from Step 3:  $u_2+u_4=\eta_1$ and $u_2+u_4=2\eta_2$. 
By Section~\ref{singularcases}, Case A, 
all $3$-fold crossings in $y(\bu)$ occur along spines with both wide and 
 narrow $3$-fold crossings (Figure~\ref{worm}(a)). In Step 3 above, the spine is a $3$-grid line.  
The type $013$ in Step 3 are wide crossings, so the narrow crossings are type $234$. These occur 
in $g_\bZr(\bu)$ in the cases where the $3$-grid line goes through the lower left and upper right corners of $R_\bZr$: 
$u_2+u_4=\eta_1$ and $u_2+u_4=2\eta_2$.

\smallskip

\noindent{\it Step} 5. Finally, we discuss the six types of $3$-fold crossings that
have the form $0ij$ or $1ij$ for $i,j=2,3,4$, $i< j$.  Any such crossing
occurs along an edge of $R_\bZr$.
By (\ref{gli}), a grid line $\ell_{i,k}(u_i)$ has the equation $\bv_i\cdot(x,y)=(-u_i+k)$.
So if $\ell_{i,k}(u_i)\cap\ell_{j,k'}(u_j)=\{(x,y)\}$ then $V_{i,j}(x,y)=(-u_i+k,-u_j+k')$, where $V_{i,j}$ 
is the $2\times 2$ matrix with rows $\bv_i$ and $\bv_j$.
Recalling that we identify $(x,y)$ with 
$[x,y]^t$, we define an affine map $A_{i,j,k,k'}(x,y):=M_{i,j}V_{i,j}(x,y)-(k,k')$, where
\begin{equation}\label{matrices}
M_{2,4}:=\left(\begin{matrix}1&0\\0&1\end{matrix}\right), \ \ \ M_{2,3}:=\left(\begin{matrix}1&0\\-1&-1\end{matrix}\right),  \ \ \ M_{3,4}:=\left(\begin{matrix}-1&-1\\0&1\end{matrix}\right).
\end{equation}
Then $A_{i,j,k,k'}$ maps a crossing $(x,y)$ of 
$\ell_{i,k}(u_i)$ and $\ell_{j,k'}(u_j)$
($i,j=2,3,4$, $i< j$)
to a point $(u_2,u_4)$ in the plane 
of the bifurcation diagram. The matrices $M_{i,j}$ 
correspond to the fact that $u_2+u_3+u_4=0$.  
When the map $A_{i,j,k,k'}$ is applied to the boundary 
$\partial R_\bZr$ of $R_\bZr$ the image is a parallelogram
in the $(u_2,u_4)$-plane whose intersections 
with $[0,1]^2$ are the values of $(u_2,u_4)$ 
in the bifurcation diagram that correspond
to $3$-fold crossings of type $0ij$ or $1ij$ in $g_\bZr(\bu)$. 

For each type $ij\in\{24,23,34\}$, Figure~\ref{types} shows  $A_{i,j,k,k'}(\partial R_\bZr)$ 
for each of the values of $(k,k')$ so that $A_{i,j,k,k'}(\partial R_\bZr)\cap[0,1]^2\ne\emptyset$.
The caption describes which sides correspond to type $0ij$ and $1ij$ and which $(k,k')$ are used. 

We conclude by observing that
steps 1 - 5 account for all possible $3$-fold and $5$-fold crossings 
in $g_\bZr(\bu)$, and also account for all the segments in Figure~\ref{markov}.. This shows that
the bifurcation diagram is correct.
\end{proof}

\begin{figure}[h!]
\includegraphics{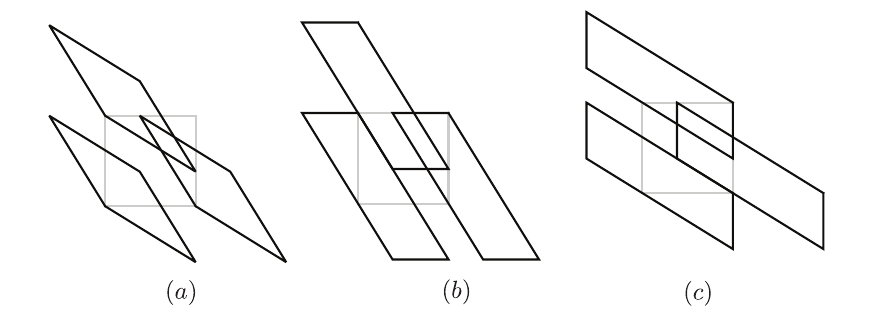}
\caption{
Images of $\partial R_\bZr$ under $A_{i,j,k,k'}$ in the $(u_2,u_4)$-plane are parallelograms (black). 
Only those choices of $k,k'$ resulting in images that meet $[0,1]^2$ (grey)
are shown. 
(a) Images under $A_{2,4,0,0}$, $A_{2,4,0,1}$, and $A_{2,4,1,0}$.
The slope  $-\gamma$  sides correspond to type $024$ and 
slope $-1/\gamma$ sides to type $124$. 
(b) Images under $A_{2,3,0,-1}$, $A_{2,3,-1,-1}$, and $A_{2,3,0,-2}$.
The slope $-\gamma$ sides correspond to type $123$, and slope $0$ sides (including 
part of $\partial [0,1]^2$)  correspond to type $023$. 
(c)  Images under $A_{3,4,-2,0}$, $A_{3,4,-1,0}$, and $A_{3,4,-1,-1}$. 
The slope $-1/\gamma$ sides correspond to type $034$ and vertical sides (including 
part of $\partial [0,1]^2$)  correspond to type $134$. 
\label{types}} 
\end{figure} 

\begin{definition}
Fix $x\in\sX$. For each $n_1\in \Z$ we define the $n_1$th \emph{bent $0$-grid line} in $x$, 
denoted $\lambda_{0,n_1}$, to be the piecewise linear curve connecting the points
$\{\bd_{(n_0,n_1)}\}_{n_0\in\Z}$ (where the points $\bd_\bn$  are as defined in Lemma~\ref{movement}).
Similarly, for each $n_0\in \Z$, we define the $n_0$th 
\emph{bent $1$-grid line} in $x$, denoted
$\lambda_{1,n_0}$, to be the piecewise linear curve connecting the vertices
$\{\bd_{(n_0,n_1)}\}_{n_1\in\Z}$.
\end{definition}

The next lemma shows that bent $0$- and $1$-grid lines $\lambda_{1,n_0}$ and $\lambda_{n_1,0}$ 
closely follow the 
corresponding grid lines $\ell_{1,n_0}$ and $\ell_{n_1,0}$.

\begin{lemma}\label{follow}
Suppose $x\in \sX_1$ and $x=y(\bu)^*$ for $y(\bu)\in Y_1$. Then 
$||\ell_{0,n_1}-\lambda_{0,n_1}||, ||\ell_{1,n_0}-\lambda_{1,n_0}||<2/3$
(in the sense of the Euclidean 
uniform closeness of the two curves).
The same results hold for singular Penrose tilings $x\in \sX\backslash \sX_1$, 
where $y(\bu)=\varphi(x)$.
\end{lemma}

\begin{proof}
For $x\in \sX_1$, this follows from Lemma~\ref{movement}. For 
$x\in \sX\backslash \sX_1$, we let $x_k=y(\bu_k)^*\to x$, where $y(\bu_k)\in Y_1$
(see Lemma~\ref{closure})
can be chosen so that $\ell_{0,n_1}(u_0)=\ell_{0,n_1}(u^k_0)$
and $\ell_{0,n_1}(u_0)=\ell_{0,n_1}(u^k_0)$ for all $||\bn||$ small. 
\end{proof}

The bent $0$- and $1$-grid lines  
 $\{\lambda_{1,n_0}(x)\}_{n_0\in\Z}$ and $\{\lambda_{0,n_1}(x)\}_{n_1\in\Z}$
partition $\R^2$ into a tiling by tetragons (i.e., quadrilaterals):

\begin{definition}\label{tetragons} 
For $x\in\sX$ and $\bn=(n_0,n_1)\in\Z^2$, let  $U_\bn(x)$ be the tile bounded by the tetragon 
with vertices $\bd_\bn$, $\bd_{\bn+\be_0}$, $\bd_{\bn+\be_0+\be_1}$,  and $\bd_{\bn+\be_1}$,
(the edges of $U_\bn(x)$ are segments of the edge curves $\lambda_{1,n_0}$, $\lambda_{1,n_0+1}$,
$\lambda_{0,n_1}$ and $\lambda_{0,n_1+1}$ in $x$). 
We call $U_\bn(x)$ an \emph{unmarked tetragon Penrose tile},
and call $u(x):=\{U_\bn(x)\}_{\bn\in\Z^2}$ an \emph{unmarked tetragon Penrose tiling}.
\end{definition}

There is a unique 
tetragon tile $U_\bn(x)$ in $u(x)$ corresponding to each Wang patch $w_\bn(x)$ in $x$ (see 
Figure~\ref{tetragon}).
In Section~\ref{pwp} we will show that there are, up to translation, $11$ different tetragon prototiles
$\{U^b:b=0,\dots,10\}$, so some of the tetragon tiles 
correspond to more than one 
Wang patche $w^a$.  
 We will return to these tilings in Section~\ref{pwp}.

\begin{figure}[h!]
\begin{center}
\includegraphics{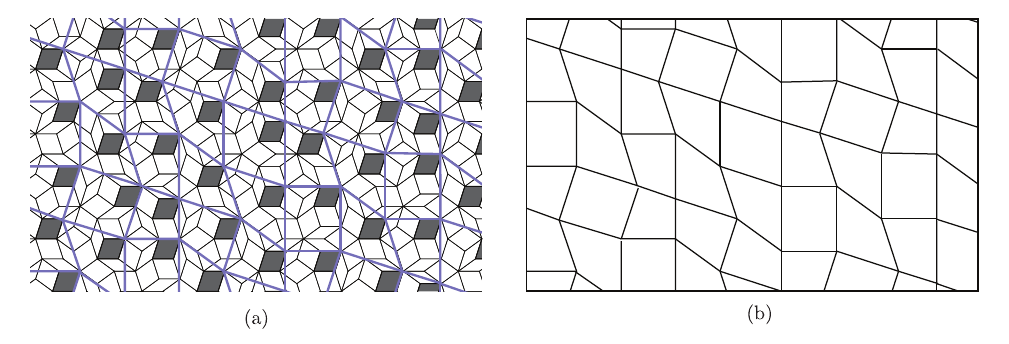}
\caption{Swatch of Penrose tiling $x\in\sX$ (a) with type 
$\bv_0'\wedge\bv'_1$ tiles $D_\bn$ shaded
to show corners of Wang patches $w^\bn(x)$.
(b) Edge curves connect the upper right vertices $\bd_\bn$ of $D_\bn$ delineating
unmarked tetragon Wang tiles $U_\bn$.  
\label{tetragon}}
\end{center}
\end{figure}

\begin{definition}\label{sigma1}
For each  Wang patch $w^a$, $a\in \cA$, 
we define  its \emph{symbol}
$\sigma(w^j)=z_0\otimes z'_0\in\{0\otimes 0,0\otimes 1,1\otimes 0,1\otimes 1\}$
as follows: If a single direction $\bv_2^\perp$
trail segment passes through $w^j$, we put $z'_0=0$, but 
if two distinct direction $\bv_2^\perp$
trail segments pass through $w^j$, we put $z'_0=1$.
Similarly, if 
a single direction $\bv_4^\perp$ 
trail segment passes through $w^j$, we put $z_0=0$, but 
if two distinct direction $\bv_4^\perp$
trail segments pass through $w^j$, we put $z_0=1$,
(compare Definition~\ref{sigma}).
\end{definition}

\begin{lemma}\label{readsturmian}
Let $x\in \sX$ be a Penrose tiling covered by Wang patches $\{w_\bn(x)\}_{\bn\in\Z^2}$. 
Then $z\otimes z'=\{\sigma(w_{\bn}(x))\}_{\bn\in \Z^2}$
is a $\Z^2$ Sturmian sequence. 
Let $u_2=\phi(z)$, $u_4=\phi(z')$,   and let $\bu_0=(0,0,u_2,\{-(u_2+u_4)\},u_4)$.  If $x\in \sX_1$ then 
$y(\bu_0)\in Y_1$ and $x=T^{\bt_0} y(\bu_0)^*$ for some unique  $\bt_0\in\R^2$ with $||\bt_0||<4/3$.
If $x\in \sX\backslash \sX_1$ then $y(\bu_0)^*$ has an unfilled worm or unfilled cartwheel, and 
there is a unique  $\bt_0\in\R^2$ with $||\bt_0||<4/3$ so that 
$x$ is $T^{\bt_0} y(\bu)^*$ with the worm or cartwheel filled.
\end{lemma}

\begin{proof}
First let $x=y(\bu)^*\in \sX_1$, so that $y(\bu)\in Y_1$. By Definitions~\ref{sigma}~and~\ref{sigma1},  
 $\sigma(g_{\bn}(\bu)^*)=\sigma(g_{\bn}(\bu))$ for all $\bn\in\Z^2$. Thus $\{\sigma(w_\bn(x))\}_{\bn\in\Z^2}$
 is a $\Z^2$ Sturmian sequence since $\{\sigma(g_\bn(\bu))\}_{\bn\in \Z^2}$ is $\Z^2$ Sturmian
by  Lemma~\ref{read}. For $\bu_0$ as described 
in the lemma, $y(\bu_0)$ and $y(\bu)$ are translates since they read the 
same Sturmian sequence.
Thus there is a unique $\bt_0$ so that 
$y(\bu)=T^{\bt_0}y(\bu_0)$. 
For $y(\bu)$ we have $\bZr\in R_\bZr$, whereas  for $y(\bu)^*$ the lower left vertex of $R_\bZr$
is $\bZr$. Thus $||\bt_0||\le ||\bbf_0+\bbf_1||<4/3$, and (\ref{equi}) implies $x=T^{\bt_0} y(\bu_0)^*$.

If $x\in\sX\backslash\sX_1$ let $x_n\to x$ in $\sX$. Let $(z\otimes z')_n$ be the  
$\Z^2$ Sturmian sequence for $x_n$, and let $\bu_{0,n}$ be the value of $\bu_0$ (Lemma~\ref{read})
for $(z\otimes z')_n$. Then the definition of the tiling metric $\rho$ 
provides that $x_n\to x$ implies $(z\otimes z')_n$ converges to 
a $\Z^2$ Sturmian sequence $z\otimes z'$, which is the sequence read from $x$,
(i.e., $z_{n_0}\otimes z'_{n_1}=\sigma(w_{(n_0,n_1)}(x))$. 
If $\bu=y^{-1}(\varphi(x))$ then since $x\in\sX\backslash\sX_1$, 
$y(\bu)^*$ has either an unfilled worm or an unfilled cartwheel.
But also $y(\bu)=T^{\bt_{0,n}}y(\bu_{0,n})$  for some $\bt_{0,n}$ with $||\bt_{0,n}||<4/3$, and 
in fact $\bt_{0,n}=\bt_0$ for all sufficently large $n$.
So  $y(\bu)=T^{\bt_0}y(\bu_0)$, but also 
$T^{\bt_n}y(\bu_{0,n})^*=T^{\bt_0}y(\bu_{0,n})^*\to x$.
Thus $x$ is obtained from $y(\bu)^*$ by filling in a worm or cartwheel.
\end{proof}

 \subsection{The proof of Theorem~\ref{mainthm}}

The next two propositions, together with Proposition~\ref{firstpart}, when used in conjunction with  Proposition~\ref{exp},
complete the proof of Theorem~\ref{mainthm}.

\begin{proposition}\label{secondpart}
Let $x\in \sX_1$ be a nonsingular Penrose tiling and let $\cV\subseteq\R^2$ be a direction in $\R^2$ 
not perpendicular to either $\bv_0$ or $\bv_1$. Then there is an $r_1>0$ (independent of $x$ and $\cV$) so that 
if $x[\cV^{r_1}]=x'[\cV^{r_1}]$ for any (not necessarily nonsingular) Penrose tiling $x'\in \sX$, then $x=x'$.
\end{proposition}

\begin{proposition}\label{thirdpart}
Let $x\in \sX\backslash \sX_1$ be a singular Penrose tiling and let $\cV\subseteq\R^2$ be a direction in $\R^2$ 
not perpendicular to any 5th root of unity. Then there is an $r_2>0$ (independent of $x$ and $\cV$) so that 
if $x[\cV^{r_2}]=x'[\cV^{r_2}]$ for any Penrose tiling $x'\in \sX$, then $x=x'$.
\end{proposition}

The proofs of these Propositions are similar, but we separate them for conceptual clarity. 
We begin with the following two lemmas.

\begin{lemma}\label{trans}
There exists an $r_0>0$ so that the following holds. Suppose $x\in \sX$ with $x=T^{\bt_0}x'$ for $||\bt_0||<4/3$.
If $x[B_{r_0}(\bZr)]=(T^{\bt'_0}x')[B_{r_0}(\bZr)]$ for $||\bt'_0||<4/3$ then $\bt'_0=\bt_0$.
\end{lemma}

This is essentially the locally free property for $\sX$, but on a new scale.

\begin{proof}
For the unmarked Penrose prototiles $p_5$ we compute $\Delta_{p_5}=2/5$, (see Definition~\ref{tmetric}). 
For any prototile set $p$, $\Delta_{\gamma p}=\gamma{\Delta_p}$,
so $\Delta_{\gamma^3 p_5}>4/3$.
Consider the MLD topological conjugacy $(C^{-1})^3:\sX\to \gamma^3\sX$ where 
$\gamma^3\sX\subseteq X_{\gamma^3 p_5}$. Let $r_0>0$ be the radius for the local map $(C^{-1})^3$,
(see Proposition~\ref{local}). Now assume 
$(T^{\bt'_0}x')[B_{r_0}(\bZr)]= x[B_{r_0}(\bZr)]=(T^{\bt_0}x')[B_{r_0}(\bZr)]$.
Let $x_3:=(C^{-1})^3 x\in\gamma^3\sX$.
By conjugacy,  $(T^{\bt'_0} x'_3))[B_{r_0}(\bZr)]=(T^{\bt_0}x'_3))[B_{r_0}(\bZr)]$.
Since $||\bt_0||,||\bt'_0||<4/3<\Delta_{\gamma^3 p}$, Lemma~\ref{useful}(1) shows $\bt'_0=\bt_0$.
\end{proof}

\begin{lemma}\label{twos}
There is an $r'_1>0$ so that for all $x\in \sX_1$: 
\begin{equation}\label{ae1}
\forall n_0\in \Z\text{\rm\ and }\bs_0\in \lambda_{1,n_0},\exists
n'_1\in \Z\text{\rm\  so\ that }w_{(n_0,n_1)}(x)\subseteq B_{r'_1}(\bs_0),
\end{equation}
and 
\begin{equation}\label{ae2}
\forall n_1\in \Z\text{\rm\ and }\bs\in \lambda_{0,n_1},\exists
n'_0\in \Z\text{\rm\  so\ that }w_{(n_0,n_1)}(x)\subseteq B_{r'_1}(\bs_).
\end{equation}
\end{lemma}

\begin{proof}
There is an $r'_1>0$ so that for any 
$\bn=(n_0,n_1)\in \Z^2$ the following holds: 
If $\bs_0\in \lambda'_{0,n_1}=\overline{\bd_\bn\,\bd_{\bn+\be_1}}$ 
and $\bs_1\in \lambda'_{1,n_0}=\overline{\bd_\bn\,\bd_{\bn+\be_0}}$,
then $w_\bn(x)\cup w_{\bn+\be_1}(x) \subseteq B_{r'_1}(\bs_0)$
and $w_\bn(x)\cup w_{\bn+\be_0}(x) \subseteq 
B_{r'_1}(\bs_1)$.
Thus for any $n_0\in \Z$ and 
$\bs\in\lambda_{1,n_0}(x)$,
 $B_{r'_1}(\bs)$ contains at least one Wang patch
from the row $\{w_{(n_0,n_1)}(x) \}_{n_1\in\Z}$.
This proves (\ref{ae1}), and the proof of   (\ref{ae2}) is the same.
\end{proof}

\begin{proof}[Proof of Proposition~\ref{secondpart}]
For each $\bn=(n_0,n_1)\in\Z^2$, let $z_{n_0}\otimes z'_{n_1}:=\sigma(w_\bn(x))$. 
By the assumption that $x\in\sX_1$,
Lemma~\ref{readsturmian} implies that 
$z\otimes z'\in Z\otimes Z$ is a $\Z^2$ Sturmian sequence,
Moreover,  Lemma~\ref{readsturmian} implies
that if $u_2=\phi(z)$, $u_4=\phi(z')$ and 
$u_3=\{-(u_2+u_4)\}$,  then for $\bu_0=(0,0,u_2,u_3,u_4)$ we have  
$x=T^{\bt_0} y(\bu_0)^*$ for $y(\bu_0)\in Y_1$ and (unique)
$\bt_0\in\R^2$ with $||\bt_0||<4/3$.

Let $r_1\ge\max\{r_0,r'_1\}$,  (Lemmas~\ref{trans} and \ref{twos}). 
The proof now proceeds by showing that we can find $z\otimes z'$ and $\bt_0$
from $x[\cV^{r_1}]$.

Since $\cV$ is not perpendicular to $\bv_0$ or $\bv_1$, it 
crosses every $0$-grid line $\ell_{0,n_1}$ and every $1$-grid line 
$\ell_{1,n_0}$ in $y(\bu)$. Thus by  Lemma~\ref{follow}, $\cV$ crosses every 
bent $0$-grid line  $\lambda_{0,n_1}$ and every bent $1$-grid line 
$\lambda_{1,n_0}$ in $x$.

For each $n_0\in \Z$, we use (\ref{ae1}) to find $\bs_0\in\cV\cap\lambda_{1,n_0}$
and $n'_1\in \Z$ so that $w_{(n_0,n'_1)}(x)\subseteq B_{r_1}(\bs_0)\subseteq \cV^{r_1}$.
Similarly, we use (\ref{ae2}) to find $\bs_1\in\cV\cap\lambda_{0,n_1}(x)$
and $n'_0\in \Z$ so that $w_{(n'_0,n_1)}(x)\subseteq B_{r_1}(\bs_1)\subseteq \cV^{r_1}$.
Let $J:=\{\bn\in\Z^2: w_{\bn}(x)\subseteq \cV^{r_1}\}$. 
Since $\cV^{r_1}=\bigcup_{\bs\in \cV}B_{r_1}(\bs)$, we see that 
$\pi_0(J)=\pi_1(J)=\Z$. Thus  $x[\cV^{r_1}]$ 
determines $(z\otimes z')[J]$, which by Lemma~\ref{recover}
determines $z\otimes z'$. 

Now suppose $x'=y(\bu_0)^*$ from above, so that $x=T^{\bt_0}x'$ with $||\bt_0||<4/3$. 
Let $\bt'_0\in\R^2$ satisfy $||\bt'_0||<4/3$ and suppose $x[\cV^{r_1}]= (T^{\bt'_0}x')[\cV^{r_1}]$,
(noting this holds for $\bt'_0=\bt_0$).
Then since $r_1\ge r_0$, we have $(T^{\bt'_0}x')[B_{r_0}(\bZr)]=(T^{\bt_0}x')[B_{r_0}(\bZr)]$ 
and Lemma~\ref{trans} implies $\bt'_0=\bt_0$.
\end{proof}

\begin{proof}[Proof of Proposition~\ref{thirdpart}]

The diameter of the regular decagon $D$ with edge length $2/5$ is $r'_2=(2/5)(1+\sqrt{5})$. If $D$ sits in the upper half 
plane with an edge on the $x$-axis, centered along the $y$-axis, then $D\cup (-D)\subseteq B_{r'_2}({\bZr})$. The same 
is true if one or both of the decagons is replaced with one or both of the hexagons $H_1$ and $H_2$ from 
the unfilled worms (in the upright position shown in Figure~\ref{worm}, (b)). It follows that if $\bs$ lies in the 
central axis of a worm or cartwheel in a singular Penrose tiling $x$, then $B_{r'_2}(\bs)$ contains either 
an entire tiled hexagon or tiled decagon. 

Now let $x$ be a singular Penrose tiling. Assume $r_2\ge \max\{r_1,r'_2\}$.  
We can apply the proof of Proposition~\ref{secondpart}, noting that 
$\cV$ is not perpendicular to $\bv_0$ or $\bv_1$, to see that $x[\cV^{r_2}]$ determines 
a tiling $y(\bu)^*$ which becomes $x$ after filling either (a) an unfilled worm 
or (b) an unfilled  cartwheel. 
Thus it remains to show that the direction of filling the worm ($+$ or $-$)
or the cartwheel ($0,1,\dots,8$ or $9$) can be determined
from $x[\cV^{r_2}]$.

By assumption, $\cV$ is not perpendicular 
to any 5th root of unity $\bv_j$, whereas the central axis 
$l$ of a worm, or a cartwheel spoke $l_j$ is perpendicular to some $\bv_j$.
In case (a) there is a single worm along an axis $l$. Let  
$\bs$ be the point where $l$ and $\cV$ intersect. Then 
by the choice of $r_2$
there is at least one hexagon $H$ in the  
unfilled worm in $y(\bu)^*$ satisfying
$H\subseteq B_{r_1}(\bs)\subseteq \cV^{r_2}$. That worm is filled in 
$x[\cV^{r_2}]$, and the direction of that filling determines how the worm is filled in $x$.

Case (b) is the case that $x$ has a cartwheel. There are two sub-cases: (i) $B_{r_2}(\bs)\subseteq\cV^{r_2}$ contains an 
unfilled decagon in $y(\bu)^*$. This decagon is filled in $x$, and that filling 
determines the direction of filling for all of the spokes.
(ii) $\cV$ crosses all five spokes, $l_j$, $j=0,1,\dots,4$.
If $\bs_j$ is the intersection of $l_j$ with $\cV$, then 
there is a hexagon $H^j\subseteq B_{r_2}(\bs_j)\subseteq\cV^{r_2}$ in $y(\bu)^*$. 
The filling of each $H^j$ determines the filling of the spoke it is in, which, in turn,  
determines the filling of the decagon. 
\end{proof}

\section{The Penrose Wang tiles}\label{subwang} \label{pwp}

\subsection{Wang tiles}\label{swangtiles}

For $n,m> 1$, let $\cA=\{0,\dots,n-1\}$ be a finite alphabet, and let 
$\cC_{{\rm b},{\rm t}}=\{0,..,m_{{\rm b},{\rm t}}-1\}$ 
and $\cC_{{\rm l},{\rm r}}=\{0,..,m_{{\rm l},{\rm r}}-1\}$ 
be two finite sets of \emph{colors}.
Formally, a \emph{Wang tile set} is $q:=\{Q^a:a\in\cA\}$ where, formally, each 
\emph{Wang tile} $Q^a$ is a tuple
$Q^a=(a,c^a_{\rm b},c^a_{\rm t},c^a_{\rm l},c^a_{\rm r})$ 
with $c^a_{\rm b},c^a_{\rm t}\in \cC_{{\rm b},{\rm t}}$ and $c^a_{\rm l},c^a_{\rm r}\in \cC_{{\rm l},{\rm r}}$.
More concretely, we view the Wang tile $Q^a$ as a unit square tile (shape  $\be_0\wedge\be_1$) 
with the \emph{label}\footnote{
Putting labels on a Wang tiles is not conventional, but it allows two different Wang tiles to have the same 
edge
colors. This makes it possible to realize the $\Z^2$ full shift with Wang tiles.
}
 $a$, and \emph{colored edges}   $c^a_{\rm b},c^a_{\rm t},c^a_{\rm l},c^a_{\rm r}$: bottom, top, left and 
 right. 
 We will interpret a Wang tile set $q$ as an FLC prototile set
by defining $q^{(2)}$ to be the condition that the tiles must meet edge to edge, 
and the colors of adjacent edges must match, and we 
 let $W_q$
be the corresponding FLC tiling space. 
If $W_q\ne\emptyset$ (see Section~\ref{history}), then $(W_q,\R^2,T)$
is an FLC tiling dynamical system that we call a \emph{Wang tile dynamical system}. 

Alternatively, we can interpret a Wang tile set 
$q$ as a recipe for $1$-step $\Z^2$ SFT $(Z_q,\Z^2,S)$, called a  \emph{Wang tile $\Z^2$ SFT}. 
For this, we define
$Z_q\subseteq \cA^{\Z^2}$ in terms of horizontal (and vertical) permitted words
$\cP_0$ (and $\cP_1$) 
where $a\,a'\in\cP_0$ if and only if $c^a_{\rm r}=c^{a'}_{\rm l}$ 
(and $(a\,a')^t\in\cP_1$ (i.e., $a'$ below $a$), if and only if $c^a_{\rm b}=c^{a'}_{\rm t}$).

\begin{remark}\label{makeSFT}
There is a 1:1 correspondence between $Z_q\times [0,1)^2$ and $W_q$, in which  $(z,\bu')$ with
$z=(a_\bn)_{\bn\in\Z^2}\in Z_q$, and $\bu'\in[0,1)^2$ corresponds to 
$\{Q^{a_\bn}+\bn+\bu'\}_{\bn\in\Z^2}\in W_q$.
\end{remark}

Now suppose that $(Z,\Z^2,S)$ is a 
$1$-step $\Z^2$ SFT, $Z\subseteq\cA^{\Z^2}$, defined by the horizontal 
and vertical permitted words $\cP_0$  
and $\cP_1$ (see Section~\ref{subshifts}). We say $(Z,\Z^2,S)$ is \emph{colorable}
if for some finite sets $\cC_{{\rm b},{\rm t}}$ and $\cC_{{\rm l},{\rm r}}$, there are functions 
(i)  $\xi_{\rm l},\xi_{\rm r}:\cA\to \cC_{{\rm l},{\rm r}}$ so that $aa'\in \cP_0$ if and only if 
$\xi_{\rm r}(a)=\xi_{\rm l}(a')$, and 
(ii) $\xi_{\rm b},\xi_{\rm t}:\cA\to \cC_{{\rm b},{\rm t}}$ so that $(aa')^t\in \cP_1$ if and only if 
$\xi_{\rm b}(a)=\xi_{\rm t}(a')$.  Clearly, if $(Z_q,\Z^2,S)$ is the $1$-step $\Z^2$ SFT that comes from
 a Wang tile set $q$, then it is colorable. Conversely, suppose that $(Z,\Z^2,S)$ is a colorable $1$-step SFT. 
For each $a\in \cA$ define the Wang tile 
$Q^a:=(a,\xi_{\rm b}(a),\xi_{\rm t}(a),\xi_{\rm l}(a),\xi_{\rm r}(a))\in q$. 
Then $(Z_q,\Z^2,S)$ is identified with $(Z,\Z^2,S)$ by the map that reads symbols. Note that our formulation
of Wang tiles here  
allows for $Z_q=\cA^{\Z^2}$ to be a full shift.

\begin{proposition}
Let $(W_q,\R^2,T)$ be a Wang tile dynamical system and let 
$(Z_q,\Z^2,S)$ be the corresponding $\Z^2$ SFT.  
Then a direction $\cV\subseteq\R^2$ is expansive for $(W_q,\Z^2,T)$
if and only if $\cV$ is expansive for $(Z_q,\Z^2,S)$. 
The same result holds for any $\Z^2$ subshift $Z\subseteq Z_q$, 
and the corresponding tiling space $W\subseteq W_q$.
 \end{proposition}
 
 \begin{proof}This follows from Lemma~\ref{asashift}, Definition~\ref{exp0}, Proposition~\ref{exp}, and Theorem~\ref{FLCdirectioncor}.
\end{proof}

\subsection{Brief history}\label{history}
Wang tiles were first described by Hao Wang in 1961 \cite{Wang}, who posed 
the question, now known as \emph{Wang's problem}:  is it possible to determine algorithmically 
(from $q$ and $q^{(2)}$) 
whether or not $W_q$ is nonempty? 
Wang proved (see \cite{GS}) that the problem is solvable in the class of Wang tile sets $q$ such 
 that any $W_q\ne\emptyset$ contains a doubly periodic tiling $x$
(i.e., $T^{n_0\be_0} x=x$ and $T^{n_1\be_1} x=x$ for some 
$n_0,n_1>0$), and he conjectured that all Wang tile sets $q$ belong to this class.
Soon thereafter, however,
Berger \cite{Berger}  showed 
that there exists a Wang tile set $q$  with $W_q\ne \emptyset$ 
such that every $x\in W_q$ is aperiodic.
In such a case, it is customary to say $q$ is an \emph{aperiodic Wang tile set}.
Moreover, Berger  \cite{Berger} showed that  
in general, Wang's problem is undecidable.

In terms of dynamics, aperiodicity  for the Wang tile set $q$
means that $(W_q,\R^2,T)$ is free.
Berger's \cite{Berger}  apreiodic set $q$ of Wang tiles 
has $\#(q)>20,\!000$, which led to a race to reduce this number 
(see \cite{GS} for a nice account).
Recently, Jeandel and Rao \cite{JR} ended this race by constructing 
an aperiodic Wang tile set $q_{\rm jr}$ with  $\#(q_{\rm jr})=11$ 
and also showing that $\#(q_{\rm jr})=11$ is the smallest possible size.

Lebb\'e showed \cite{sebastian1} that the Jaendel/Rao SFT $(Z_{\!q_{\rm jr}},\Z^2,S)$ is uniquely ergodic but not minimal \cite{sebastian1}. However, $(Z_{\!q_{\rm jr}},\Z^2,S)$ has a unique 
minimal subshift $(Z_0,\Z^2,S)$, and that subshift has 
$\N_1(S)=\{0,\gamma+3,-3\gamma+2,-\gamma+5/2\}$, \cite{sebastian}, (but
{\it a priori}, the set $Z_{\!q_{\rm jr}}\backslash  Z_0$ could be responsible for additional non-expansive directions).
Finally, Lebb\'e showed that $(Z_0,\Z^2,S)$ is almost automorphic, with 
$\Sigma_S=(\Z[\alpha]/\Z)\times(\Z[\alpha]/\Z)$,
the same as Raphael Robinson's Penrose Wang tile SFT, (Section~\ref{prwt}). 
However, the two are not topologically conjugate.

\subsection{Suspensions, compositions, deformations, and 
continuous orbit equivalence}\label{ke}\ \\
Given a  dynamical system 
$(Z,\Z^2,S)$, let  ${\widetilde Z}=Z\times [0,1]^2/\sim$, where 
$(z,(1,s_1))\sim(S^{\be_0}z,(0,s_1))$ and  $(z,(s_0,1))\sim(S^{\be_1}z,(s_0,0))$
for  $z\in Z$ and $\bs:=(s_0,s_1)\in[0,1]^2$. 
The \emph{unit suspension} of $(Z,\Z^2,S)$ 
is the dynamical system
$({\widetilde Z},\R^2,{\widetilde S})$ defined 
 by ${\widetilde S}^\bt(z,\bs)=(S^{\lfloor \bt\rfloor}z,\{\bt+\bs\})$. For a Wang tiling system $(W_q,\R^2,T)$, where $q=\{Q^j:j\in\cA\}$, 
and $(Z_q,\Z^2,S)$ is the corresponding SFT, clearly
$(\widetilde{Z_q},\R^2,\widetilde{S})$  is topologically conjugate to $(W_q,\R^2,T)$.

Next, we change the shape of the Wang tiles $q$ from squares to a parallelograms. Let  
$\bbf_0,\bbf_1\in\R^2$ with $0<\angle\bbf_0\bbf_1<\pi$, and let $A=[\bbf_0;\bbf_1]^t$.
The \emph{parallelogram Wang tiles} $Aq:=\{AQ^a:Q^a\in q\}$ are defined to be the copies 
of $\bbf_0\wedge\bbf_1=A(\be_0\wedge\be_1)$ marked the same way as the
squares in $q$. Following Remark~\ref{makeSFT}, we 
write $v\in W_{Aq}$ by $v=\{AQ^{a_\bn}+A(\bn-\bu')\}_{\bn\in \Z^2}$, where $\bu'\in[0,1)^2$ and  $(a_\bn)_{\bn\in\Z^2}\in Z_q$.
Here $(Aq)^{(2)}:=A(q^{(2)})$. 

A parallelogram Wang tiling dynamical system $(W_{Aq},\R^2,T)$ is no longer
the unit suspension
of $(Z_q,\Z^2,S)$.
However, 
we can recover it
from $({\widetilde {Z_q}},\R^2,{\widetilde S})$ as follows:
For an $\R^2$-action $(X,\R^2,T)$, and a matrix $A\in Gl(2,\R)$, we define the \emph{composition action}
by $(T\circ A)^\bt x:=T^{A\bt}x$.

\begin{lemma}\label{composed}
\
\begin{enumerate}
\item The parallelogram Wang tile system  $(W_{Aq},\R^2,T)$ is topologically conjugate 
to the composition action $({\widetilde {Z_q}},\R^2,{\widetilde S}\circ A)$. 
\item The parallelogram Wang tile system  $(W_{Aq},\R^2,T)$ is continuously orbit equivalent (Definition~{\rm \ref{equivalences}}) to the Wang tile system $(W_{q},\R^2,T)$.
\item  The (weakly and strongly) 
expansive directions for  $(W_{Aq},\R^2,T)$  are the directions
$A(\cV)$ for  the   
expansive directions $\cV$ of $(Z_q,\Z^2,S)$.

\end{enumerate}
\end{lemma}

\begin{proof}
For (1), we construct a model for ${\widetilde S}\circ A$ on $Z_q\times(\bbf_0\wedge\bbf_1)$. The conjugacy is the 
homeomorphism ${\rm id}\times A:Z_q\times [0,1)^2\to  Z_q\times(\bbf_0\wedge\bbf_1)$. 
Part (2) follows from part (1), and the fact 
$(W_q,\R^2,T)$ and $(W_q,\R^2,T\circ A)$, 
(which is topologically conjugate to $(W_{Aq},\R^2,T)$), both act on  
$W_q$, and have the same orbits. 
Part (3) follows from Propositions~\ref{exp} and \ref{FLCdirectioncor}.
\end{proof}

\begin{remarks}\label{rcomposed}
(See e.g. \cite{RRS}):
\begin{enumerate}
\item Suppose that $\widetilde T$ is the unit suspension of $T$. Then $\Sigma_{\widetilde T}=\pi^{-1}\Sigma_T$
where $\pi(\bs):=\{\bs\}:\R^2\to\T^2$.
\item The point spectrum of an $\R^2$  composition action $T\circ A$ satisfies $\Sigma_{T\circ A}=A^t\Sigma_T$.
 \end{enumerate}
\end{remarks}

\begin{remark}\label{deformation}
In the terminology of \cite{SW} and \cite{CS}, a
\emph{deformation} $\psi$ is a continuous orbit equivalence 
(i.e., orbit preserving homeomorphism)  
between two FLC tiling spaces 
such that the restriction to each tile depends only on its prototile type.
The point of view in \cite{SW} is closer to ours here\footnote{The point of view in \cite{CS} is in terms of the 
Anderson-Putnam complex.}, 
and in particular, it is shown in \cite{SW}  that any 
$\R^2$ FLC tiling space $(X,\R^2,T)$ is 
continuously orbit equivalent to the unit suspension of 
$\Z^2$ subshift, essentially
realized as a subspace of a square Wang tiling space
(any full shift can be realized by Wang tiles). 
Some examples of deformation are discussed in Section~\ref{prwt}.
\end{remark}

\subsection{The Penrose Wang tiles}\label{prwt}

From now on, $\cA:=\{0,\dots,23\}$.
In Section~\ref{dgp}  we discussed the $24$ Wang patches 
$\{w^a:a\in\cA\}$, shown in Figure~\ref{wangpatches}. 
Here we will construct the corresponding $24$ \emph{Penrose Wang tiles} due to
Raphael Robinson (\cite{GS}, Exercise 11.1.2), 
made by deforming the  Wang patches, as well as some related tiles and tilings.

For each the bottom and top margin ($s\in\{{\rm b},{\rm t}\}$)
of each Wang patch $w^a$, we assign a finite word 
\emph{edge code} ${\beta}^a_s\in\{\2,\3,4\}^*$
that records 
the tile edges along the outside of that margin. 
Similarly, for each left and right margin ($s\in\{{\rm l},{\rm r}\}$)
we  assign an edge code ${\beta}^a_s\in\{2,\3,\4\}^*$.
We record, 
$j$ for $\bv'_j$ and $\overline{j}$ for $-\bv'_j$, skipping the corner tiles ($j=0$ or $j=1$) at either end. 
The edges are read left to right 
for $\beta_{\rm b}^a$ and $\beta_{\rm t}^a$, 
and up for $\beta_{\rm l}^a$ and $\beta_{\rm r}^a$. 
For example, since the edges along the bottom margin in $w^0$ are 
$\bv'_0,\bv'_4,-\bv'_3,-\bv'_2 ,\bv'_0$, we record  $\beta^0_{\rm b}=4\overline{3}\overline{2}$.
The edge codes are shown, for each $a\in \cA$, in the second column of Table~\ref{codes}.

\begin{table}
\begin{tabular}{||c||c||c||c||c||}
\hhline{||=||=||=||=||=||}
$a\in \cA$
&$\beta^a_{\rm b},\beta^a_{\rm t},\andd \beta^a_{\rm l},\beta^a_{\rm r}$
&$c^a_{\rm b},c^a_{\rm t} \andd c^a_{\rm l},c^a_{\rm r}$
&$v \text{ for }\bb^v_{\rm b},\bb^v_{\rm t}\andd
\bb^v_{\rm l},\bb^v_{\rm r}$&$t\text{ for }U^t$\\
\hhline{||=||=||=||=||=||}
\hline
0&$4\3\2, 4\2,\  2\3\4, 2\4$, &0,1, 0,1&0,1, 0,1&0\\
\hline
1&$ {2\4,2\3\4},\ \42, \4\32$&	2,3, 2,3& 1,0, 1,0&1\\
\hline
2& $\2, \3\2\ , \4, \3\4$ &4,5, 4,5&2,3, 2,3&2\\
\hline
3&$\2\3, \2,\ \4\3, \4$ &6,4,  6,4 &3,2, 3,2&3\\
\hline
4&	$\3\24, \2\34,\  \3\42, \4\32$& 7,3, 7,3&0,0, 0,0&4\\
\hline
5&	$4\3\2, 4\2\3,\ 2\3\4, 2\4\3$ &0,8, 0,8	&0,0, 0,0&4\\ 	
\hline
6&	$4\2, \3\24,\ 2\4, \3\42$	&1,7, 1,7	&1,0, 1,0&1\\ 
\hline
7&	$4\2\3, \24,\ 2\4\3, \42$&	8,2, 8,2&	0,1, 0,1&0\\ 
\hhline{||=||=||=||=||=||}
8&	$\2\34, 4\2\3,\ \3\4, \4\3$&	3,8, 5,6&	0,0, 3,3&5\\ 
\hline
9&$   
{\3\24}, 4\3\2,\ \3\4, \4\3$&	7,0, 5,6	&0,0, 3,3&5\\ 
\hline
10&	$\3\2, \2\3,\ \4\32, 2\4\3$&	5,6, 3,8	&3,3, 0,0&6\\ 
\hline
11&	$\3\2, \2\3,\ \3\42, 2\3\4$&	5,6, 7,0	&3,3, 0,0&6\\
\hhline{||=||=||=||=||=||}
12&	$\2\3, \2,\ \4\32, \42$&	6,4, 3,2	&3,2, 0,1&7\\
\hline
13&	$\2, \3\2,\ 2\4, \42\3$&	4,5, 1,0	&2,3, 1,0&8\\
\hline
14&	$\2\34, \24,\ ,\4\3, \4$&	3,2, 6,4	&0,1, 3,2&9\\ 
\hline
15&	$4\2, 4\3\2,\ \4, \3\4$&	1,0, 4,5	&1,0, 2,3&10\\ 
\hhline{||=||=||=||=||=||}
16&	$\24, \3\24,\ 2\4, \4\32$& 	2,7, 1,3	&1,0, 1,0&1\\ 
\hline
17&	$4\2\3, 4\2,\ 2\3\4, \42$&	8,1, 0,2	&0,1, 0,1&0\\
\hline
18&	$4\3\2, \24,\ 2\4\3, \24$&	0,2, 8,1	&0,1, 0,1&0\\
\hline
19&	$4\2, \2\34,\ \42, \3\42$&	1,3, 2,7	&1,0, 1,0&1\\
\hhline{||=||=||=||=||=||}
20&	$\2, \2\3,\ \42, 2\3\4$&	4,6, 2,0	&2,3, 1,0&8\\ 
\hline
21&	$\3\2, \2,\ \4\32, 2\4$&	5,4, 3,1	&3,2, 0,1&7\\ 
\hline
22&	$\24, 4\3\2,\ \4, \4\3$&	2,0, 4,6	&1,0, 2,3&10\\ 
\hline
23&	$\2\34, 4\2,\ \3\4, \4$	&3,1, 5,4	&0,1, 3,2&9\\ 
\hhline{||=||=||=||=||=||}
\end{tabular}
\caption{Columns 1-3: edge codes for Wang patches and corresponding Wang tile edge colors. 
Columns 4 and 5: edge vector numbers for tetragon Wang tiles, and tetragon Wang tile type.  
\label{codes}}
\end{table}

Column 2 of Table~\ref{codes} shows 
that there are nine distinct top and bottom code words, and for each of them we assign a 
\emph{color} 
$c_{\rm b}, c_{\rm t}\in \cC_{{\rm b}, {\rm t}}=\{0,1,\dots,8\}$. 
Similarly, there are nine distinct left and 
right code words each assigned a color 
$c_{\rm l}, c_{\rm r}\in \cC_{{\rm l}, {\rm r}}=\{0,1,\dots,8\}$. 
The colors are shown in Table~\ref{codes}, Column 3.

\begin{definition}
For each row  in Table~\ref{codes}, (each $a\in\cA$), we define a \emph{Penrose Wang tile} 
$Q^a:=(a,c_{\rm b},c_{\rm t},c_{\rm l},c_{\rm r})$, where  
$c_{\rm b},c_{\rm t},c_{\rm l},c_{\rm r}$ are the colors shown in the Column 3.
The  $24$ \emph{Penrose Wang tiles} are given by 
$q_{\rm p}=\{Q^a:a\in \cA\}$, where $q_{\rm p}^{(2)}$ is 
as described for Wang tiles in general in Section~\ref{swangtiles}.
\end{definition}

Two tiles $Q^a$ and $Q^{a'}$ can 
be vertically (or horizontally) adjacent in some Wang tiling in $W_{q_{\rm p}}$ 
if and only if the
same is true for the corresponding Wang patches $w^a$ and $w^{a'}$ for some $x\in\sX$.  
The  \emph{Penrose Wang tile} dynamical system is denoted 
by $(W_{q_{\rm p}},\R^2,T)$, and the corresponding 
$\Z^2$ SFT is denoted by 
$(Z_{q_{\rm p}},\Z^2,S)$, where we write $(a_\bn)_{\bn\in\Z^2}\in Z_{q_{\rm p}}$.

\begin{corollary}\label{free}
Let $\{Q^{a_\bn}+\bn+\bu'\}_{\bn\in \Z^2}\in  W_{q_{\rm p}}$, $\bu'\in [0,1)^2$, $(a_\bn)_{\bn\in\Z^2}\in Z_ {q_{\rm p}}$, 
be an arbitrary Penrose Wang tiling, 
$($see Remark~{\rm \ref{makeSFT}}$)$.
Then there is a 
Penrose tiling $x\in\sX$, unique up to translation, so that 
$w_\bn(x)$ is type $w^{a_\bn}$ for all $\bn\in\Z^2$. 
Conversely, for each Penrose tiling $x\in\sX$ there is a Penrose Wang tiling 
$\{Q^{a_\bn}+\bn\}_{\bn\in\Z^2}$, unique up to translation, 
so that $w_\bn(x)$ is type $w^{a_\bn}$ for each $\bn\in\Z^2$.
In particular, $(W_{q_{\rm p}},\R^2,T)$ and $(\sX,\R^2,T)$ are continuously orbit equivalent
(via a deformation), and   
 $(W_{q_{\rm p}},\R^2,T)$ and  $(Z_{q_{\rm p}},\Z^2,S)$
are nonempty and free.
\end{corollary}

\begin{definition}
We define \emph{Penrose parallelogram Wang tiles} as follows:
For $\bbf_0$ and $\bbf_1$ in (\ref{rhombs}), let $A=[\bbf_1;\bbf_0]^t$, so that 
$A[0,1]^2=\bbf_1\wedge\bbf_0=R$, the rhombus that supports the grid patches. 
The \emph{Penrose parallelogram Wang tiles} are the tiles $Aq_{\rm p}=\{AQ^a:a\in\cA\}$,
which  
differ from the square Penrose Wang tiles $q_{\rm p}$ only in their shape. 
\end{definition}

Like $(W_{q_{\rm p}},\R^2,T)$, the Penrose parallelogram Wang tile
dynamical system 
$(W_{Aq_{\rm p}},\R^2,T)$ is nonempty, free, and 
is Kakutanni equivalent to $(\sX,\R^2,T)$.
More of its properties are described in Theorem~\ref{pwt} below.

\medskip

Next, we recall the unmarked Penrose tetragon tilings 
$u=\{U_\bn(x)\}_{\bn\in\Z^2}$, 
described in Definition~\ref{tetragons}. 
For any $x\in\sX$, the vertices 
are $\bd_\bn$, $\bd_{\bn+\be_0}$, $\bd_{\bn+\be_0+\be_1}$ and $\bd_{\bn+\be_1}$.
We denote the vectors along the four sides by 
\begin{align*}
\bb_{\rm b}&:=\bd_{\bn+\be_0}-\bd_\bn, \ \ 
&\bb_{\rm t}:=\bd_{\bn+\be_0+\be_1}-\bd_{\bn+\be_1},\\
 \bb_{\rm l}&:=\bd_{\bn+\be_1}-\bd_\bn,\text{\ \ \ \ \ \ \ \ \ \ and } 
 &\bb_{\rm r}:=\bd_{\bn+\be_0+\be_1}-\bd_{\bn+\be_0}.
 \end{align*}  
 These vectors can also be found by summing the edge vectors along the 
sides of  the corresponding Wang patch $w_\bn(x)$: 

\begin{lemma}\label{type}
Suppose that $U_\bn(x)$ is the unmarked tetragon Penrose tile corresponding to the 
Wang patch $w_\bn(x)$ in some $x\in \sX$
$($we think of $U_\bn$ as type $U^a$ if  the type of $w_\bn(x)$ is $w^a$\,$)$. 
Then for each $s\in\{{\rm b},{\rm t}\}$ and $s'\in\{{\rm l},{\rm r}\}$,
\begin{equation}\label{sums}
\bb_{s}^a=\bv'_0+\sum_{j\in \beta_s^a}\bv'_j\text{ and } 
\bb_{s'}^a=\bv'_1+\sum_{j\in \beta_s^a}\bv'_j.
\end{equation}
In particular, $\bb_{s}^a$ for  $s\in\{{\rm b},{\rm t},{\rm l},{\rm r}\}$ depends only
on the code word $\beta_{s}^a$, or equivalently, the color $c_s^a$ {\rm (Table~\ref{codes}, Column 2)}.
\end{lemma}

Computing the sums (\ref{sums}) for each $a\in \cA$ shows that there
are four possible vectors for $\bb_{\rm b}^a$ and $\bb_{\rm t}^a$,
and four more for $\bb_{\rm l}^a$ and $\bb_{\rm r}^a$.
These are shown in Columns 2 and 5 on Table~\ref{codes1} as
$\bb_{\rm b}^v,\bb_{\rm t}^v$ for $v\in\{0,1,2,3\}$,
and as $\bb_{\rm l}^{v'},\bb_{\rm r}^{v'}$ for $v'\in\{0,1,2,3\}$. 
Columns 3 and 6 of Table~\ref{codes1} show all the (Wang tile) colors
$c_{\rm b},c_{\rm t},c_{\rm l}$, or $c_{\rm r}$ that correspond
to these vectors, as described in Lemma~\ref{type}.

\begin{figure}[h!]
\begin{center}
\includegraphics{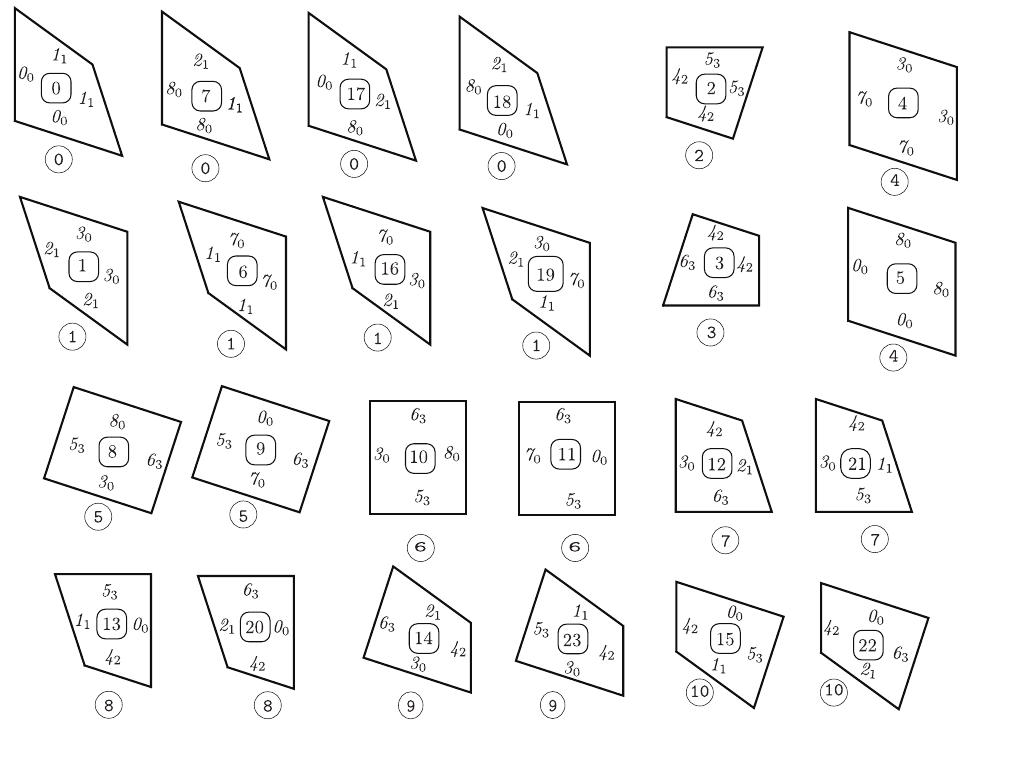}
\caption{The 24 tetragon Penrose tiles 
$p_4:=\{M^a:a\in\cA\}$, where $a$ is shown in the rounded squares. Each edge is labeled with its
color, $c_{\rm b}, c_{\rm t}\in\{{\it 0},\dots {\it 8}\}$ and 
$c_{\rm l}, c_{\rm r}\in\{{\it 0},\dots {\it 8}\}$, (Table~\ref{codes}, Column 3), with the 
subscript showing the side vector $v$ or $v'$ 
(Table~\ref{codes1}, Columns 2 and 4). The tile shape (the unmarked 
tetragon tile type $v$) is shown in the circle below each tile. The 24 Penrose Wang 
 tiles (or Penrose parallelogram Wang tiles) may be obtained by deforming each tetragon 
with the unit square or rhomb.  
\label{tetragon2}}
\end{center}
\end{figure}

Column 4 of Table~\ref{codes} shows which edge vector defines each of the four sides of the unmarked 
tetragon Penrose tile. 
Table~\ref{codes}, Column 5 shows that there are 11 different unmarked 
tetragon Penrose tiles $U^t$ for $t\in\{0,\dots,10\}$.

\begin{table}
\begin{tabular}{||c|c|c||c|c|c||}
\hhline{||=|=|=||=|=|=||}
$v$ &$\bb^v_{\rm b}, \bb^v_{\rm t}$ 
&$c_{\rm b},c_{\rm t}$&$v' $&$\bb^{v'}_{\rm l}, \bb^{v'}_{\rm r}$ 
& $c_{\rm l},c_{\rm r}$\\
\hhline{||=|=|=||=|=|=||}
0&$(1.1710, -0.3804)$&0,3,7,8    &0&$(0, 1.231)$&0,3,7,8\\
\hline
1&$(0.8472, -0.6155)$&1,2&1&$(-0.3236, 0.9960)$&1,2\\
\hline
2&$(0.7236, -0.2351)$&4&2&$(0, 0.7608)$&4\\
\hline
3&$(1.047, 0)$&5,6&3&$(0.3236, 0.9960)$&5,6\\
\hhline{||=|=|=||=|=|=||}
\end{tabular}
\caption{
The vertical and horizontal edge vectors for marked and unmarked tetragon Penrose tiles. 
The vectors for $v=0$ and $v=2$, and  for $v'=0$ and $v'=2$, have the same slopes.
Columns 3 and 5 show which edge colors correspond to each vector.
\label{codes1}}
\end{table}

\begin{definition}\label{twt}
We define \emph{marked tetragon Penrose tiles} $m_4=\{M^a,a\in \cA\}$  
as follows. For each $a\in \cA$ (Table~\ref{codes}, Column 1) we define 
$M^a$ by decorating the  
corresponding unmarked tetragon Penrose tile $U^v$, where $v=v(a)$ is 
given in Table~\ref{codes}, Column 5, by coloring its edges using the Wang tile colors
 $c_{\rm b},c_{\rm t},c_{\rm l}$, and $c_{\rm r}$ (Table~\ref{codes}, Column 3).
The matching rule  $m_4^{(2)}$ is the requirement that the colors of the adjacent 
edges of two adjacent tiles must match.
\end{definition}

Any two horizontal or vertical edges that have the same color also have the same edge vector.
Thus the adjacency condition is the same as for the Penrose Wang tiles.  
We call $(X_{m_4},\R^2,T)$ the 
\emph{tetragon Penrose tiling dynamical system}.
We see  that 
$(X_{m_4},\R^2,T)$ 
is continuously orbit equivalent
to $(W_{q_{\rm p}},\R^2,T)$  by 
 a deformation, so 
in particular, $(W_{q_{\rm p}},\R^2,T)$
 is nonempty and free.

The next two results summarize the dynamical properties of the four dynamical systems studied in this section. Theorem~\ref{pwt} discusses 
the \emph{tetragon Penrose tiling} system $(X_{m_4},\R^2,T)$  
and \emph{Penrose parallelogram Wang tile} system $(W_{Aq_{\rm p}},\R^2,T)$.
Corollary~\ref{pwt2} discusses
Raphael Robinson's \emph{Penrose Wang tile} system $(W_{q_{\rm p}},\R^2,T)$, and the corresponding $\Z^2$ SFT   
$(Z_{q_{\rm p}},\Z^2,S)$. 

\begin{theorem} \label{pwt}
The dynamical systems 
$(X_{m_4},\R^2,T)$  and  $(W_{Aq_{\rm p}},\R^2,T)$ 
are  both topologically conjugate $($MLD in the first case$)$ to the Penrose tiling dynamical system 
$(\sX,\R^2,T)$.
Each is strictly ergodic and almost automorphic, with  
$\Sigma_T=Z[e^{2\pi i/5}]\subseteq\R^2$, $\mathfrak{M}_T=\{1,2,10\}$,
and $\N_1(T)=\{\bv_0^\perp, \dots,\bv_4^\perp\}$.
\end{theorem}

\begin{corollary} \label{pwt2}
The dynamical systems 
 $(W_{q_{\rm p}},\R^2,T)$ and  $(Z_{q_{\rm p}},\Z^2,S)$ are both
strictly ergodic, and almost automorphic, satisfying  
$\Sigma_T=\Z[\alpha]\times \Z[\alpha]\subseteq\R^2$
and $\Sigma_S=(\Z[\alpha]/\Z)\times (\Z[\alpha]/\Z)\subseteq \T^2$.
Also $\mathfrak{M}_T=\mathfrak{M}_S=\{1,2,10\}$,
and $\N_1(T)=\N_1(S)=\{0,1,\infty, \gamma, \alpha\}$. 
\end{corollary}

\begin{proof}
This follows from Theorem~\ref{pwt}, Lemma~\ref{composed}, and Remark~\ref{rcomposed}. 
However, $\Sigma_S=(\Z[\alpha]/\Z)\times (\Z[\alpha]/\Z)$ is easier to see from the fact that 
$R_\alpha\otimes R_\alpha$ is the Kronecker factor of $S$.
\end{proof}

\begin{remark}
In \cite{mann}, the minimality of $(Z_{q_{\rm p}},\Z^2,S)$ is stated as a conjecture, and an earlier preprint version  
of \cite{mann}. The preprint \cite{mann} 
also proves that  $\N_1(T)$ is as stated above.
\end{remark}

\begin{proof}[Proof of Theorem~\ref{pwt}]
It suffices to prove the topological conjugacies, and then everything else follows from Theorems~\ref{RPenrose} and 
\ref{mainthm}. 
For  $(X_{m_4},\R^2,T)$ there is an obvious 
MLD topological conjugacy.

We will now construct a topological conjugacy 
$\Phi:\sX\to W_{Aq_{\rm p}}$.
For $x\in \sX$, let $\varphi(x)=y(\bu)\in Y$ and  for $\bu=(u_0,\dots,u_4)\in \T^5_0$, let 
$\bu'=(u_0,u_1)\in[0,1)^2$ (as in (\ref{tess})).
For each $\bn\in\Z^2$, let $a_\bn\in \cA$ be such  that 
$w_\bn(x)$ is a type $w^{a_\bn}$ Wang patch. Define 
$\Phi(x):=\{AQ^{a_\bn}+A(\bn-\bu')\}_{\bn\in\Z^2}$. 
Note that $\Phi$ is well defined: each $AQ^{a_\bn}+A(\bn-\bu')$ in $\Phi(x)$ is 
supported on $R_\bn=R+A(\bn-\bu')$, and  
$\Phi(x)$ satisfies  the local rule $(Aq_{\rm p})^{(2)}$.
Clearly $\Phi$ is continuous. 
 
Next, we consider a different formula for 
the restriction $\Phi|_{\sX_1}$ of $\Phi$ to $\sX_1$. 
For $x\in \sX_1$, let $y(\bu)=\varphi(x)\in Y_1$. 
By Lemma~\ref{nns}, 
$w_\bn(x)=g_\bn(\bu)^*$ for all $\bn\in\Z^2$. 
Define $\theta:Y_1\to W_{Aq_{\rm p}}$ to be the map 
that replaces each grid patch $g_\bn(\bu)$ in $y(\bu)$ with the Wang tile  
$\Theta(g_\bn(\bu)):=  
AQ^{a_\bn}+A(\bn-\bu')$, noting 
that both $g_\bn(\bu)$ and $AQ^{a_\bn}+A(\bn-\bu')$
are supported on $R_\bn$. In particular, $\theta$ is a local map in the sense of  
Definition~\ref{local}, so it 
is uniformly continuous, and satisfies $\theta(T^\bt y(\bu))=T^\bt(\theta(y(\bu)))$. Clearly 
$\theta(\varphi(x))=\Phi|_{\sX_1}(x)$ (from the first part), so $\Phi$ is uniformly continuous on $\sX_1$, and satisfies
$\Phi|_{\sX_1}(T^\bt x)=T^\bt(\Phi|_{\sX_1}(x))$. It follows that $\Phi$ is the unique continuous  
extension of $\theta(\varphi(x))$, and thus satisfies $\Phi(T^\bt x)=T^\bt(\Phi(x))$.

It remains to show $\Phi$ is a bijection. 
Lemma~\ref{free} shows there is a 1:1 correspondence between $T$ orbits 
on $\sX$ and $T$ orbits on $W_{Aq_{\rm p}}$.
 But since $T$ acts freely  
 on both $\sX$ and $W_{Aq_{\rm p}}$, the fact that $\Phi(T^\bt x)=T^\bt(\Phi(x))$ shows that $\Phi$ is bijective on each orbit as well. 
 This implies $\Phi$ is a bijection. 
\end{proof}

\end{document}